\documentclass[11pt]{article}
\usepackage{amsmath, amssymb, amsthm, graphicx, hyperref}
\usepackage{tocloft}

\newtheorem{thm}{Theorem}[section]
\newtheorem{prop}[thm]{Proposition}
\newtheorem{lem}[thm]{Lemma}

\newcommand{\prob}{\mathbb{P}}
\newcommand{\expec}{\mathbb{E}}
\newcommand{\var}{\operatorname{Var}}
\newcommand{\demph}{\textbf}
\newcommand{\R}{\mathbb{R}}
\newcommand{\N}{\mathbb{N}}
\newcommand{\Z}{\mathbb{Z}}
\newcommand{\C}{\mathbb{C}}

\newcommand{\re}{\textrm{Re}}
\newcommand{\im}{\textrm{Im}}

\newcommand{\wzeta}{\omega}
\newcommand{\cint}{\int\limits}
\newcommand{\mellin}[1]{{\mathcal M}[#1]}

\title{On the number of $n$-dimensional representations of $SU(3)$, the Bernoulli numbers, and the Witten zeta function}

\author{Dan Romik\footnote{Department of Mathematics, University of California, Davis, One Shields Ave, Davis, CA 95616, USA. Email: \texttt{romik@math.ucdavis.edu}.}
}

\date{\today\footnote{Revised from earlier versions of March 11, 2015; April 16, 2015; and September 27, 2015.}}

\begin{document}

\maketitle

\begin{abstract}
We derive new results about properties of the Witten zeta function associated with the group $SU(3)$, and use them to prove an asymptotic formula for the number of $n$-dimensional representations of $SU(3)$ counted up to equivalence. Our analysis also relates the Witten zeta function of $SU(3)$ to a summation identity for Bernoulli numbers discovered in 2008 by Agoh and Dilcher. We give a new proof of that identity and show that it is a special case of a stronger identity involving the Eisenstein series.
\end{abstract}

\renewcommand{\thefootnote}{\fnsymbol{footnote}} 
\renewcommand{\thefootnote}{\arabic{footnote}} 

\section{Introduction}

\subsection{Asymptotic enumeration of representations}

\label{subsec:asym-rep}

Let $p(n)$ denote the number of $n$-dimensional representations of the group $SU(2)$, counted up to equivalence. Since $SU(2)$ has (up to equivalence) one irreducible representation $V_k$ of each dimension $k=1,2,\ldots$, each $n$-dimensional representation $\oplus_{k=1}^\infty r_k V_k$ is encoded by a unique integer partition $n=\lambda_1+\ldots+\lambda_m$, such that $r_k$ is the number of parts in the partition that are equal to $k$. Thus $p(n)$ is the number of integer partitions of $n$. For large $n$, $p(n)$ is given approximately by the well-known Hardy-Ramanujan asymptotic formula \cite{hardy-ramanujan} (see also \cite[Chapter~2]{newman})
\begin{equation} \label{eq:hardy-ramanujan}
p(n) = (1+o(1))\frac{1}{4\sqrt{3}n} e^{\pi \sqrt{2n/3}} \qquad \textrm{as }n\to\infty.
\end{equation}

This work began with the idea of answering a similar question for the group $SU(3)$.
Let $r(n)$ denote the number of $n$-dimensional representations of the group $SU(3)$, counted up to equivalence, or equivalently the number of $n$-dimensional complex-linear representations of the Lie algebra $A_2=sl(3,\C)$. The irreducible representations of $SU(3)$ are a family of representations $W_{j,k}$ indexed by pairs of integers $j,k\ge1$, where it is well known (see, e.g., \cite[Chapter 5]{hall}) that $\dim W_{j,k}=\tfrac12 j k(j+k)$. A general $n$-dimensional representation decomposes as a sum of these $W_{j,k}$'s, each with some multiplicity. Thus, it is easy to see that the numbers $r(n)$ satisfy the generating function identity
\begin{equation} \label{eq:rofn-genfun}
\sum_{n=0}^\infty r(n)x^n = \prod_{j,k=1}^\infty \frac{1}{1-x^{jk(j+k)/2}},
\end{equation}
(with the natural convention that $r(0)=1$)
analogous to Euler's product formula $\sum_{n=0}^\infty p(n)x^n=\prod_{k=1}^\infty (1-x^k)^{-1}$ for the generating function of partitions. Using \eqref{eq:rofn-genfun} one may easily compute the first few terms of the sequence $r(n)$, which are (starting at $n=0$) $1,1,1,3,3,3,8,8,9,17,19,21,35,39,
\ldots$.

One of our main results is the following analogue for the group $SU(3)$ of the Hardy-Ramanujan asymptotic formula \eqref{eq:hardy-ramanujan}.

\begin{thm}
\label{thm:su3-asym}
As $n\to\infty$, we have
\begin{equation}
\label{eq:su3-asym}
r(n) = (1+o(1)) \frac{K}{n^{3/5}} \exp\left( A_1 n^{2/5} - A_2 n^{3/10} - A_3 n^{1/5} - A_4 n^{1/10}
\right),
\end{equation}
for constants $K, A_1,A_2,A_3,A_4$, which are defined as follows. 
Let $\Gamma(s)$ and $\zeta(s)$ denote the Euler gamma function and the Riemann zeta function, respectively. 
Denote
\begin{align}
X &= \left(\tfrac{1}{9} \, \Gamma\left(\tfrac13\right)^2 \, \zeta\left(\tfrac53\right)\right)^{3/10}, \label{eq:defx} \\
Y &= -\sqrt{\pi} \zeta\left(\tfrac12\right) \zeta\left(\tfrac32\right).
\label{eq:defy}
\end{align}
Then $A_1, A_2, A_3, A_4$ are given by
\begin{align}
A_1 &= 5X^2 \ \ \ \ \ \ \ \ \ \ = 6.858260476163126\ldots, \label{eq:A1formula} \\[5pt]
A_2 &= \displaystyle X^{-1}Y \ \ \ \ \ \ \ \,=
5.773601745105114\ldots,
\label{eq:A2formula} \\[4pt]
A_3 &= \displaystyle \frac{3}{80} X^{-4}Y^2
\ \ \  = \displaystyle 0.911341072572436\ldots, \label{eq:A3formula} \\[10pt]
A_4 &= \displaystyle \frac{11}{3200} X^{-7}Y^3
= 0.351637541558209\ldots,
\label{eq:A4formula}
\end{align}
and the multiplicative constant $K$ is given by
\begin{equation} \label{eq:multiplicativeconst-def}
K = \frac{2\sqrt{3\pi}}{\sqrt{5}} X^{1/3} \exp\left(-\frac{1}{2560} X^{-10} Y^4\right)
= 
2.4462903348641789\ldots
\end{equation}
\end{thm}
The formula \eqref{eq:multiplicativeconst-def} for $K$ is a simplified version of a more complicated formula that appeared in the preprint version of this paper. The simplification depends on a result from a recent paper by Borwein and Dilcher \cite{borwein-dilcher}---see Theorems~\ref{thm:wzeta-deriv0} and \ref{thm:borwein-dilcher} and the remark following Theorem~\ref{thm:borwein-dilcher} below.

\subsection{The zeta function of $SU(3)$}

\label{subsec:zetasu3}

Our efforts to derive and prove the asymptotic formula \eqref{eq:su3-asym} led us to make several additional discoveries that are of independent interest. Understanding the asymptotic behavior of the sequence $r(n)$ turns out to require a close study of a natural Dirichlet series (or zeta function) associated with the representation theory of $SU(3)$, namely
\begin{equation} \label{eq:dirichlet-series} 
\wzeta(s) = \sum_{j,k=1}^\infty (2\dim W_{j,k})^{-s} = \sum_{j,k=1}^\infty \frac{1}{\left(j k(j+k)\right)^s}.
\end{equation}
This function has an interesting history dating back to the 1950's papers of Mordell \cite{mordell} and Tornheim \cite{tornheim}. It is (up to the trivial factor $2^{-s}$) a special case of a so-called \demph{Witten zeta function}, which is more generally a Dirichlet series $\zeta_{\mathfrak{g}}(s)$ associated with a semisimple Lie algebra $\mathfrak{g}$ and defined by
$$\zeta_{\mathfrak{g}}(s) = \sum_\rho (\dim \rho)^{-s},$$ 
where the sum extends over the irreducible representations of $\mathfrak{g}$. (The same function in the context of groups has been referred to by some authors as the \textbf{representation zeta function}; see, e.g., \cite{avni-etal}, \cite{larsen-lubotzky}.)

Much of the work to date on $\wzeta(s)$ has concerned its evaluation at positive integers. This was motivated to a large part by Mordell's discovery \cite{mordell} that the even integer values $\wzeta(2n)$ are rational multiples of $\pi^{6n}$, being given explicitly by the formula
\begin{equation}
\label{eq:mordell-bernoullipoly}
\wzeta(2n) = \frac{(-1)^{n+1} (2\pi)^{6n}}{6((2n)!)^3} \int_0^1 B_{2n}(x)^3\,dx
\qquad (n\ge1),
\end{equation}
where $B_m(x)$ denotes the $m$th Bernoulli polynomial. 
Later 
authors have since
extended Mordell's results in several ways. Witten \cite{witten} related the values at even integers $\zeta_\mathfrak{g}(2n)$ $(n\ge1)$ of the zeta function of a general semisimple Lie algebra $\mathfrak{g}$ to the volumes of certain moduli spaces of vector bundles of curves. His formulas imply that $\zeta_\mathfrak{g}(2n)$ is a rational multiple of $\pi^{2rn}$, where $r$ is the rank of the Lie algebra $\mathfrak{g}$. Subbarao and Sitaramachandrarao
\cite{subbarao-etal} proved another explicit formula for $\wzeta(2n)$, namely
\begin{equation} \label{eq:wzeta-evenvalues}
\wzeta(2n) =\frac43 \sum_{k=0}^n \binom{4n-2k-1}{2n-1} \zeta(2k) \zeta(6n-2k)
\qquad (n\ge1).
\end{equation}
According to Zagier's paper \cite{zagier}, the same formula was apparently rediscovered by Garoufalidis and Zagier and independently by Weinstein. In the same paper, Zagier also gives without proof an analogous formula for the odd integer values $\wzeta(2n+1)$, namely
$$
\wzeta(2n+1) = -4 \sum_{k=0}^n \binom{4n-2k+1}{2n} \zeta(2k) \zeta(6n-2k+3) \qquad (n\ge0).
$$
The same formula as well as its proof was published a short time later by Huard, Williams and Nan-Yue \cite{huard-williams-nanyue}, who seem unaware of Zagier's paper. The formula \eqref{eq:wzeta-evenvalues} is also given in the paper of Gunnels and Sczech \cite{gunnels-sczech} along with an analogous formula for the values of the zeta function of $SU(4)$ at even positive integers.

As an illustration of the above formulas, here are the values $\wzeta(m)$ for $1\le m\le 5$:
\begin{align*}
\wzeta(1) &= 2\zeta(3), \\
\wzeta(2) &= \frac{1}{2835} \pi^6, \\
\wzeta(3) &= -2 \pi^2 \zeta(7) + 20 \zeta(9), \\
\wzeta(4) &= \frac{19}{273648375} \pi^{12}, \\
\wzeta(5) &= -\frac29 \pi^4 \zeta(11) - \frac{70}{3}\pi^2 \zeta(13) + 252 \zeta(15).
\end{align*}
Several of the papers mentioned above also contain formulas of a similar flavor for the values of more general ``multiple zeta values'' defined in terms of several integer-valued parameters; see also \cite{gangl-etal}, \cite{ihara-etal}, \cite{ohno-zagier}, \cite{zagier-annals} for related results.

The results described above may not make evident why it is natural to consider $\wzeta(s)$ as a true Dirichlet series (that is, as a function of a complex variable $s$), but plenty of precedents from the history of analytic number theory suggest that this is worth doing. Perhaps guided by such considerations, \nobreak{Matsumoto} \cite{matsumoto} initiated the study of the analytic continuation of the three-variable Dirichlet series
\begin{equation} \label{eq:matsumoto-mordell-tornheim}
\zeta_{MT}(s_1,s_2,s_3) = \sum_{j,k=1}^\infty j^{-s_1} k^{-s_2} (j+k)^{-s_3},
\end{equation}
which contains $\wzeta(s)$ as the specialization $s_1=s_2=s_3=s$,
and several other multiple Dirichlet series of a somewhat similar nature. He referred to $\zeta_{MT}(s_1,s_2,s_3)$ as the \demph{Mordell-Tornheim zeta function}. He and his collaborators later extended such results to more general (multivariate) Witten zeta functions \cite{matsumoto-semisimple1}, \cite{matsumoto-semisimple2}, \cite{matsumoto-semisimple3}, \cite{matsumoto-semisimple4}. 

Following in the footsteps of Matsumoto's pioneering work, and guided by a specific need to understand the complex-analytic properties of $\wzeta(s)$ in connection with our asymptotic analysis of the sequence $r(n)$, we proved several results about $\wzeta(s)$. The first one is as follows.

\begin{thm}
\label{thm:analytic-continuation}
\textnormal{\textbf{(1)}}
The series \eqref{eq:dirichlet-series} converges precisely for complex $s\in\C$ satisfying $\re(s)>2/3$, 
and defines a holomorphic function in that region. 

\medskip\noindent \textnormal{\textbf{(2)}}
$\wzeta(s)$ can be analytically continued to a meromorphic function on $\C$.

\medskip\noindent \textnormal{\textbf{(3)}}
$\wzeta(s)$ possesses a simple pole at $s=\tfrac23$ with residue $\frac{1}{2\pi\sqrt{3}} \Gamma\left(\tfrac13\right)^3$, and for each $k=0,1,2,\ldots$ it has a simple pole at $s=\tfrac12-k$ with residue $(-1)^k 2^{-4k} \binom{2k}{k} \zeta\left(\frac{1-6k}{2}\right)$. It has no other singularity points.

\medskip\noindent \textnormal{\textbf{(4)}}
$|\wzeta(s)|$ grows at most at a polynomial rate for fixed $\re(s)$ as $\im(s)\to\pm \infty$. More precisely, for any closed interval $I=[a,b]\subset \R$, there exist constants $C,M>0$ such that the bound 
\begin{equation} \label{eq:wzeta-polybound}
|\wzeta(\sigma+i t)| \le C |t|^M
\end{equation}
(where $i=\sqrt{-1}$) holds for all $\sigma\in I$ and $|t|\ge 1$.
\end{thm}

We note that Matsumoto already proved part (2) of Theorem~\ref{thm:analytic-continuation} (see Theorem~1 in his paper \cite{matsumoto}), but he was studying the more general function \eqref{eq:matsumoto-mordell-tornheim} of several complex variables and did not examine more closely the properties of $\wzeta(s)$.

At the end of his discussion of the Mordell-Tornheim zeta function, \nobreak{Matsumoto} writes: ``\textit{It is now an interesting problem to study the properties of the values of $\zeta_{MT}(s_1,s_2,s_3)$ at non-positive integers.}'' We solve this problem for the case of our function $\wzeta(s)=\zeta_{MT}(s,s,s)$. The result is as follows.

\begin{thm} 
\label{thm:wzeta-specialvalues}
$\wzeta(s)$ has the values
\begin{align}
\wzeta(0) &= \tfrac13, \label{eq:wzeta-at-zero} \\
\wzeta(-n) &= 0 \qquad (n=1,2,3,\ldots). \label{eq:wzeta-trivialzeros}
\end{align}
\end{thm}

Note that the equations \eqref{eq:wzeta-at-zero}--\eqref{eq:wzeta-trivialzeros} imply that Mordell's property that $\wzeta(2n)$ is a rational multiple of $\pi^{6n}$ holds true for \emph{all} integer $n$. Furthermore, somewhat surprisingly, the formula 
\eqref{eq:wzeta-evenvalues} 
of Subbarao and Sitaramachandrarao still holds true for $n=0$ if one accepts the convention that $\binom{-1}{-1}=1$, even though their proof clearly does not apply in that case. (Perhaps less strikingly, the formula also holds true for negative values of $n$, simply because the range of the summation becomes empty.) By comparison, 
Mordell's formula \eqref{eq:mordell-bernoullipoly} gives
the incorrect value of $-1/6$ in the case $n=0$.

Our final result on $\wzeta(s)$ is a formula expressing its derivative at $s=0$ in terms of familiar special constants and a curious definite integral. 

\begin{thm} 
\label{thm:wzeta-deriv0}
Let $\gamma$ denote 
the Euler-Mascheroni constant.
The value $\wzeta'(0)$ is given by
\begin{align} 
\wzeta'(0) &= \frac{1}{12}(1+\gamma) + \frac34 \log(2\pi) - 2 \zeta'(-1) + \frac12 \int_{-\infty}^\infty \frac{\zeta\left(\tfrac32+it\right)\zeta\left(-\tfrac32-it\right)}{\left(\tfrac32+it\right) \cosh(\pi t)}\,dt
\nonumber \\[3pt] &= 1.83787706640934548356\ldots.
\label{eq:wzeta-deriv0}
\end{align}
\end{thm}

Note that the integral in \eqref{eq:wzeta-deriv0} converges very rapidly due to the exponential decay of the factor $1/\cosh(\pi t)$, and is easy to evaluate numerically to high precision.
Curiously, the numerical value of the integral (including the multiplicative factor $\tfrac12$) is $-0.002807659\ldots$, so its relative influence on $\wzeta'(0)$ is very small---less than a fifth of a percent.

\paragraph{Update added in revision.} Following publication of the preprint version of this paper, the following much simpler formula for $\omega'(0)$ was proved by Borwein and Dilcher \cite{borwein-dilcher}.

\medskip
\begin{thm} \label{thm:borwein-dilcher}
\quad $ \omega'(0) = \log(2\pi). $
\end{thm}

See also \cite{bailey-borwein} for related recent developments.

\medskip
Theorems~\ref{thm:analytic-continuation}--\ref{thm:borwein-dilcher} 
will all play a role in our asymptotic analysis leading up to the proof of Theorem~\ref{thm:su3-asym}. In particular, the poles at $s=2/3$ and $s=1/2$ and their residues, and the evaluation $\wzeta(0)=\tfrac13$, have a large influence on the form of the asymptotic formula \eqref{eq:su3-asym}. The evaluation of $\wzeta'(0)$ enters into the definition of the multiplicative constant $K$, which ends up emerging from our calculations in the form
\begin{equation} \label{eq:mult-const-altdef}
K = \frac{\sqrt{3}}{\sqrt{5\pi}} X^{1/3} \exp\left(-\frac{1}{2560} X^{-10} Y^4+ \wzeta'(0)\right),
\end{equation}
(clearly equivalent to \eqref{eq:multiplicativeconst-def} in view of Theorem~\ref{thm:borwein-dilcher}). The poles with negative real part and the location of the zeros \eqref{eq:wzeta-trivialzeros} do not enter into our present analysis of the asymptotic behavior of $r(n)$, but those results, aside from being interesting in their own right, will also start having an effect on the asymptotics of $r(n)$ if one should attempt to derive more detailed expansions for $r(n)$ which include additional lower order terms currently encapsulated in the $(1+o(1))$ term in \eqref{eq:su3-asym}; see the discussion of open problems in Section~\ref{sec:final-comments}.

\subsection{The Bernoulli numbers and Eisenstein series}

\label{subsec:bernoulli}

One of our surprising discoveries was that Theorem~\ref{thm:wzeta-specialvalues} has an intriguing connection to the Bernoulli numbers, and ultimately to the theory of modular forms. Recall that the Bernoulli numbers are defined by the power series expansion
$$ \frac{x}{e^x-1} = \sum_{n=0}^\infty \frac{B_n}{n!} x^n, $$
and that they are related to special values of the Riemann zeta function by the well-known equations
$\zeta(1-n)=(-1)^{n+1}\frac{B_{n}}{n}$, $\zeta(2n)=(-1)^{n+1}\frac{(2\pi)^{2n}}{2(2n)!} B_{2n}$, $(n\ge1)$. 

It is well-known that the Bernoulli numbers satisfy a host of remarkable summation identities (see \cite{arakawa-etal}, \cite{zagier-bernoulli}, and \cite[pp.~81--85]{berndt} for many examples). Our analysis of the properties of $\wzeta(s)$ led us to discover one such identity, namely the statement that for $n\ge1$ we have
\begin{equation} \label{eq:bernoulli}
\frac{B_{6n+2}}{6n+2} = -\frac{(4n+1)!}{(2n)!^2} \sum_{k=1}^n \binom{2n}{2k-1}
\frac{B_{2n+2k}}{2n+2k} \cdot \frac{B_{4n-2k+2}}{4n-2k+2}.
\end{equation}
Equivalently, in terms of the values at integer arguments of the zeta function, \eqref{eq:bernoulli} can be rewritten in either of the forms
\begin{align} 
\zeta(6n+2) &= \frac{2}{6n+1}\cdot\frac{(4n+1)!}{(2n)!^2} \sum_{k=1}^n \frac{\binom{2n}{2k-1}}{\binom{6n}{2n+2k-1}} \zeta(2n+2k)\zeta(4n-2k+2),
\label{eq:bernoulli-zeta-plus}
\\
\zeta(-6n-1) &= \frac{(4n+1)!}{(2n)!^2} \sum_{k=1}^n\binom{2n}{2k-1}\zeta(-2n-2k+1)\zeta(-2n+2k-1).
\label{eq:bernoulli-zeta-minus}
\end{align}
The identity \eqref{eq:bernoulli} turns out to have been discovered and proved earlier, being due to Agoh and Dilcher in their 2008 paper \cite{agoh-dilcher2008}. It belongs to a class of several summation identities that Agoh and Dilcher refer to as \demph{lacunary recurrences}, since they have the appealing property of expressing a given Bernoulli number in terms of a relatively small number of Bernoulli numbers of lower index; in the case of \eqref{eq:bernoulli}, $B_{6n+2}$ is expressed in terms of $B_{2n+2},B_{2n+4},\ldots,B_{4n}$. See also \cite{agoh-dilcher2007} for further discussion of lacunary Bernoulli number recurrences.

Note that \eqref{eq:bernoulli} does \emph{not} hold for $n=0$, when the left-hand side is equal to $1/12$ and the right-hand side is equal to $0$.

Being at first unaware of Agoh and Dilcher's result, we independently proved \eqref{eq:bernoulli}. This was fortunate, since our method of proof actually implies a stronger result that does not follow from the methods of \cite{agoh-dilcher2007}, and gives hints of a deeper theory at play. To formulate our new result, 
recall that the Eisenstein series of index $2k$ is an analytic function $G_{2k}(z)$ defined for $z$ in the upper half-plane $\mathbb{H}=\{x+iy\,:\,y>0\}$ by
$$ G_{2k}(z) = \sum_{(m,n)\in\Z^2\setminus\{(0,0)\}} \frac{1}{(mz+n)^{2k}} \qquad (k\ge2). $$
The function $G_{2k}(z)$ is related to the zeta value $\zeta(2k)$ (hence also to the Bernoulli number $B_{2k}$) by the well-known relation
$$ \lim_{\im(z)\to\infty} G_{2k}(z) = 2\zeta(2k), $$
usually formulated as the statement that $2\zeta(2k)$ is the constant term in the Fourier series $G_{2k}(z)=2\zeta(2k)+\sum_{n=1}^\infty c_{2k}(n)e^{2\pi i nz}$ (see \cite{apostol} for these facts and additional background on modular forms mentioned below). We are now ready to state our result.

\begin{thm} 
\label{thm:eisenstein}
The identity \eqref{eq:bernoulli-zeta-plus} is the ``shadow'' of a similar summation identity for the Eisenstein series. That is, for $n\ge1$ we have
\begin{equation} \label{eq:eisenstein}
G_{6n+2}(z) = \frac{1}{6n+1}\cdot\frac{(4n+1)!}{(2n)!^2} \sum_{k=1}^n \frac{\binom{2n}{2k-1}}{\binom{6n}{2n+2k-1}} G_{2n+2k}(z) G_{4n-2k+2}(z).
\end{equation}
\end{thm}

Let us now explain the connection between our results on $\wzeta(s)$ and the identity \eqref{eq:bernoulli}. Remarkably, we will show that the relation \eqref{eq:wzeta-trivialzeros} identifying the ``trivial zeros'' $s=-1, -2, \ldots$ of $\wzeta(s)$ is equivalent to that identity. It was precisely our study of the values $\wzeta(-n)$ which prompted the discovery of \eqref{eq:bernoulli} as a numerical observation for small values of $n$. We were then able to prove the identity (and the stronger Theorem~\ref{thm:eisenstein}) using separate algebraic methods and deduce \eqref{eq:wzeta-trivialzeros}, and also learned about the earlier work of Agoh and Dilcher. The same reasoning could work in reverse to deduce \eqref{eq:bernoulli} in yet another way if a more direct proof of \eqref{eq:wzeta-trivialzeros} could be found (e.g., based on a functional equation satisfied by $\wzeta(s)$), but we do not currently know of such a proof.

To conclude this discussion, let us comment on the connection to the theory of modular forms.
For $k\ge 0$, let $M_{2k}$ denote the vector space of modular forms of weight $2k$ for the modular group $SL(2,\Z)$. The space $M_{2k}$ is known to be of dimension $\lfloor k/6\rfloor$ if $k\equiv 1\ (\textrm{mod}\ 6)$, or $\lfloor k/6\rfloor+1$ otherwise. The Eisenstein series $G_{6n+2}$, and the products $G_{2n+2k}(z) G_{4n-2k+2}(z)$ appearing on the right-hand side of \eqref{eq:eisenstein}, are all elements of $M_{6n+2}$, and, curiously, the number of distinct terms on the right-hand side of \eqref{eq:eisenstein} (after grouping together the equal pairs 
$G_{2n+2k}(z) G_{4n-2k+2}(z)$, $G_{4n-2k+2}(z) G_{2n+2k}(z)$ obtained from symmetric indices $k$ and $n+1-k$) is precisely 
equal to $\dim M_{6n+2}$; check this separately for both even and odd values of~$n$.
Thus, the theory of modular forms seems like the proper context in which to consider \eqref{eq:eisenstein} and similar summation identities. It is natural to wonder whether there is a deeper connection between \eqref{eq:eisenstein} and the function $\wzeta(s)$; whether other connections of this type between modular forms and zeta functions can be found; and whether summation identities such as \eqref{eq:eisenstein} can be classified and better understood. See also the list of open problems in Section~\ref{sec:final-comments}.

\subsection{Methods, organization, and the companion \texttt{Mathematica} package}

The remainder of the paper is organized in roughly three parts: the first part consists of Section~\ref{sec:bernoulli} and \ref{sec:eisenstein}, in which we give a new proof of the summation identity \eqref{eq:bernoulli}, and then show how a modification of our argument yields a proof of Theorem~\ref{thm:eisenstein}. Our method is an extension of an algebraic technique introduced by Zagier \cite{zagier}, that enables us to reduce the identities to a family of more conventional binomial summation identities. These identities are proved using Zeilberger's algorithm and the method of Wilf-Zeilberger pairs.

In the second part of the paper, comprising Sections~\ref{sec:wzeta} and \ref{sec:wzeta2}, we study the properties of the function $\wzeta(s)$, and prove Theorems~\ref{thm:analytic-continuation}, \ref{thm:wzeta-specialvalues}, and \ref{thm:wzeta-deriv0}. The analysis relies heavily on an integral representation for $\wzeta(s)$ derived by Matsumoto \cite{matsumoto}.

In the final part, which consists of Sections~\ref{sec:asym-genfuns}--\ref{sec:proof-su3-asym}, we apply the results concerning $\wzeta(s)$ to the study of the asymptotic behavior of the sequence $r(n)$. We combine Mellin transform methods with a saddle point analysis (framed in probabilistic language as a local central limit theorem), culminating in the proof of Theorem~\ref{thm:su3-asym}. A key number-theoretic difficulty we encounter in the course of the proof is that of proving effective bounds for the decay of a certain generating function away from the saddle point; we derive the necessary bounds by using the modular transformation properties of the Jacobi theta functions. We conclude with some final comments and open problems in Section~\ref{sec:final-comments}.

A few of the proofs in the paper involve tedious computations of a purely algebraic nature, which can be checked by hand without much difficulty but are most easily verified using a computer algebra package. We prepared a companion \texttt{Mathematica} package that is available to download online \cite{romik-mathematica} and demonstrates the correctness of these computations.

\subsection{Acknowledgements}

The author thanks Craig Tracy for pointing out the papers of Matsumoto \cite{matsumoto} and Zagier \cite{zagier} and the relevance of Matsumoto's results to asymptotic analysis; Greg Kuperberg for helpful discussions; and an anonymous referee for several helpful comments. This work was supported by the National Science Foundation under grant DMS-0955584.

\section{Bernoulli numbers}

\label{sec:bernoulli}

In this section we give a new proof of the summation identity \eqref{eq:bernoulli} using a completely different method than that used by Agoh and Dilcher \cite{agoh-dilcher2008}. In the next section we show how a modification of our argument actually proves the stronger identity \eqref{eq:eisenstein} involving the Eisenstein series, which contains \eqref{eq:bernoulli} as a specialization at $z=i\infty$.

Note that we gave three equivalent formulations \eqref{eq:bernoulli}--\eqref{eq:bernoulli-zeta-minus} of the summation identity. While \eqref{eq:bernoulli} seems like the most natural number-theoretic formulation of the identity (and was the version proved by Agoh and Dilcher), the other two formulations will also play important roles, in that our proof technique will actually imply the version \eqref{eq:bernoulli-zeta-plus}, whereas we will later make use of its ``dual'' version \eqref{eq:bernoulli-zeta-minus} (the two being related via the functional equation for the Riemann zeta function) in our analysis of the function $\wzeta(s)$.

Our proof is based on adapting a technique introduced by Zagier \cite{zagier} as a way of proving the more elementary recurrence relation
$$
\zeta(2n) = \frac{2}{2n+1}\sum_{k=1}^{n-1} \zeta(2k)\zeta(2n-2k) \qquad (n\ge 2),
$$
a well-known relation due to Euler. (If one assumes the connection between the even integer zeta values to the Bernoulli numbers, this is the recurrence $B_{2n}=-\frac{1}{2n+1}\sum_{k=1}^{n-1} \binom{2n}{2k} B_{2k}B_{2n-2k}$, which has a simple proof using generating functions; see \cite[Proposition 1.15]{arakawa-etal}.) Before we proceed with the proof in the most general case, let us first illustrate the use of Zagier's technique to prove the case $n=1$ of \eqref{eq:bernoulli-zeta-plus}, which in that case becomes
$$ \zeta(8) = \tfrac{6}{7} \zeta(4)^2. $$
Define the function $f(p,q) = \frac{2}{7p^3 q^5} + \frac{3}{7p^4 q^4}+\frac{2}{7 p^5 q^3}$. One may easily verify that $f(p,q)$ satisfies the algebraic identity
$$ f(p,q) - f(p+q,q) - f(p,p+q) = \frac{6}{7 p^4 q^4}. $$
Now sum both sides of this equation over all $p,q\in\N$. The summation of the right-hand side gives $6 \zeta(4)^2/7$, and for the left-hand side we get
\begin{align*}
\sum_{p,q=1}^\infty & f(p,q) - \sum_{p,q=1}^\infty f(p+q,q) - \sum_{p,q=1}^\infty f(p,p+q) 
\\ &= \left(
\sum_{r,s=1}^\infty - \sum_{r>s\ge1} - \sum_{s>r\ge 1} 
\right)f(r,s) = \sum_{r=1}^\infty f(r,r) \\ &= \left(\frac27+\frac37+\frac27\right)\sum_{r=1}^\infty \frac{1}{r^8} = \zeta(8).
\end{align*}
For general $n$, we now try to generalize the idea by looking for a pair of functions $f_n(p,q)$, $g_n(p,q) = f_n(p,q)-f_n(p+q,q)-f_n(p,p+q)$, such that $f_n(p,q)$ is of the form
\begin{equation} \label{eq:ffun-def}
f_n(p,q) = \sum_{j=1}^{2n+1} \frac{\alpha_{n,j}}{p^{2n+j} q^{4n-j+2}},
\end{equation}
for some as-yet-undetermined coefficients $(\alpha_{n,j})_{j=1}^{2n+1}$, and such that $g_n(p,q)$ miraculously simplifies to be of the form
\begin{equation} \label{eq:gfun-simplifies}
g_n(p,q) = \sum_{k=1}^n \frac{\beta_{n,k}}{p^{2n+2k} q^{4n-2k+2}},
\end{equation}
for some coefficients $(\beta_{n,k})_{k=1}^n$. Assuming that this somewhat daring ansatz turns out to be successful, summing $g_n(p,q)$ over all $p,q\in\N$ then yields on the one hand
\begin{equation} \label{eq:sum-gofpq1}
\sum_{k=1}^n \beta_{n,k} \left(\sum_{p,q=1}^\infty p^{-2n-2k}q^{-4n+2k-2} \right) = \sum_{k=1}^n \beta_{n,k} \zeta(2n+2k) \zeta(4n-2k+2),
\end{equation}
and on the other hand
\begin{align}
& \sum_{p,q=1}^\infty f_n(p,q) - \sum_{p,q=1}^\infty f_n(p+q,q) - \sum_{p,q=1}^\infty f_n(p,p+q) 
\nonumber \\ &= \left(
\sum_{r,s=1}^\infty - \sum_{r>s\ge1} - \sum_{r>s\ge 1} 
\right) f_n(r,s)
= \sum_{r=1}^\infty f_n(r,r) \nonumber \\ &= \left(\sum_{j=1}^{2n+1} \alpha_{n,j} \right)\zeta(6n+2),
\label{eq:sum-gofpq2}
\end{align}
whereupon we obtain, after also adding a normalization condition $\sum_j \alpha_{n,j}=1$, the identity
\begin{equation} \label{eq:identity-betank}
\zeta(6n+2) 
= \sum_{k=1}^n \beta_{n,k} \zeta(2n+2k) \zeta(4n-2k+2).
\end{equation}

The difference with Zagier's original application of the technique is that in his case the undetermined coefficients were trivial (see the bottom of p.~498 of his paper \cite{zagier}), whereas here they are not. As the next lemma shows, finding them reduces to the solution of a somewhat tricky system of linear equations.

\begin{lem}
\label{lem:system-of-equations}
If $f_n(p,q)$ is defined by \eqref{eq:ffun-def}, then the function $g_n(p,q) = f_n(p,q)-f_n(p+q,q)-f_n(p,p+q)$ can be expressed in the form \eqref{eq:gfun-simplifies} if and only if the coefficients $(\alpha_{n,j})_{j=1}^{2n+1}, (\beta_{n,k})_{k=1}^n$ satisfy the equations
\begin{align}
&\sum_{j=1}^{2n+1} \alpha_{n,j}\Bigg( \binom{6n}{m-4n+j} -\binom{2n+j-2}{m-4n+j}-\binom{4n-j}{m-6n} \Bigg) \nonumber \qquad \qquad \qquad 
\\ & \hspace{160pt}= \sum_{k=1}^n \beta_{n,k} \binom{6n}{m-4n+2k}
\label{eq:lineqs}
\end{align}
for $m=2n,2n+1,\ldots,10n-2$.
\end{lem}

\begin{proof}
In \eqref{eq:gfun-simplifies}, substitute the definition of $g_n(p,q)$ and multiply both sides by $p^{6n} q^{6n} (p+q)^{6n}$. The identity becomes
\begin{align*}
\sum_{j=1}^{2n+1} \alpha_{n,j} \Bigg( & p^{4n-j} q^{2n+j-2}(p+q)^{6n}-p^{4n-j} q^{6n} (p+q)^{2n+j-2} 
\\ & \ \ - p^{6n} q^{2n+j-2} (p+q)^{4n-j} \Bigg)
= 
\sum_{k=1}^n \beta_{n,k} p^{4n-2k} q^{2n+2k-2} (p+q)^{6n},
\end{align*}
an equation in which both sides are homogeneous polynomials in $p,q$ of degree $12n-2$. Extracting a common factor of $q^{12n-2}$ and denoting $x=p/q$, the identity is clearly seen to be equivalent to the single-variable polynomial identity
\begin{align*}
&\sum_{j=1}^{2n+1} \alpha_{n,j} \Bigg( x^{4n-j}(1+x)^{6n}-x^{4n-j}(1+x)^{2n+j-2}-x^{6n}(1+x)^{4n-j}\Bigg) \qquad \qquad
\\ & 
\hspace{200pt} =\sum_{k=1}^n \beta_{n,k} x^{4n-2k} (1+x)^{6n}.
\end{align*}
Expanding both sides of this relation in powers of $x$ and equating coefficients of each monomial $x^m$ gives precisely the equations \eqref{eq:lineqs}. The power $m$ on the right-hand side ranges from $2n$ to $10n-2$; on the left-hand side it ranges from $2n-1$ to $10n-1$, but it is easy to verify that the coefficients of $x^{2n-1}$ and $x^{10n-1}$ vanish.
\end{proof}

Having derived the system of equations \eqref{eq:lineqs}, we used a computer to solve them for small values of $n$. This led us to guess based on numerical evidence that
\begin{align}
\alpha_{n,j} &= \frac{1}{6n+1} \cdot \frac{(4n+1)!}{(2n!)^2}
\frac{\binom{2n}{j-1}}{\binom{6n}{2n+j-1}},
\label{eq:coeffs-guess1} \\ 
\beta_{n,k} &= 
\frac{2}{6n+1} \cdot \frac{(4n+1)!}{(2n!)^2}
\frac{\binom{2n}{2k-1}}{\binom{6n}{2n+2k-1}} = 2\alpha_{n,2k}.
\label{eq:coeffs-guess2} 
\end{align}
In our situation we happened to already know what the coefficients $\beta_{n,k}$ should be, since we knew in advance the form of the identity \eqref{eq:identity-betank} to be proved---see Section~\ref{sec:wzeta} for the reasoning that led to its discovery. However, it is worth noting that the equations can be solved even with undetermined values for these coefficients, which can be useful for discovering additional identities; see Section~\ref{sec:final-comments} where we give two additional examples of identities discovered using this approach.

Having guessed the formulas \eqref{eq:coeffs-guess1}--\eqref{eq:coeffs-guess2},
our final task will be to verify that the guesses are correct. We start by checking the normalization condition.

\begin{lem} 
\label{lem:normalization}
If the numbers $\alpha_{n,j}$ are defined by \eqref{eq:coeffs-guess1}, then for all $n\ge1$ we have that $\sum_{j=1}^{2n+1} \alpha_{n,j}=1$.
\end{lem}

\begin{proof}
Start by observing that the definition \eqref{eq:coeffs-guess1} of the $\alpha_{n,j}$'s as an interesting triangle of integers also makes sense when $n$ is a \emph{half-integer}. It turns out that the identity $\sum_j \alpha_{n,j}=1$ is still correct in that case, and this is both easier and more natural to prove. That is, we now wish to show that
\begin{equation} \label{eq:sumequals1}
S(n) := \sum_{k=0}^n F(n,k) = 1
\end{equation}
for all integer $n\ge 0$, where we define
$$
F(n,k) = \alpha_{n/2,k+1} = \frac{1}{3n+1}\frac{(2n+1)!}{(n!)^2}\frac{\binom{n}{k}}{\binom{3n}{n+k}}.
$$
This can be proved using the method of Wilf-Zeilberger pairs \cite{wilf-zeilberger}. Define the ``proof certificate'' function
$$ R(n,k) = -\frac{k (2n-k+1) (11n^2+2k^2-5nk+27n-6k+16)}{3(n+1)(3n+2)(3n+4)(n-k+1)},
$$
which was found with the help of the \texttt{Mathematica} package \texttt{fastZeil} \cite{paule-schorn}, \cite{paule-schorn2}, a software implementation of Zeilberger's algorithm \cite{zeilberger} (see also \cite{petkovsek-wilf-zeilberger}).
Let $G(n,k)=F(n,k)R(n,k)$. Then one may verify by direct computation that
\begin{equation} \label{eq:wzpair}
F(n+1,k)-F(n,k) = G(n,k+1)-G(n,k)
\end{equation}
(divide both sides by $F(n,k)$ to turn the relation into an equation between two rational functions, or refer to the companion package \cite{romik-mathematica} for an automated verification).
Summing both sides of \eqref{eq:wzpair} over all integer $k$ shows that $S(n+1)=S(n)$. Since $S(0)=1$, \eqref{eq:sumequals1} follows by induction.
\end{proof}

\begin{lem}
\label{lem:lineqs}
The coefficients $\alpha_{n,j}$ and $\beta_{n,k}$ as defined by \eqref{eq:coeffs-guess1}--\eqref{eq:coeffs-guess2} satisfy the system of equations \eqref{eq:lineqs}.
\end{lem}

\begin{proof}
Since $\beta_{n,k}=2\alpha_{n,2k}$, it makes sense to rearrange the terms between the two sides of the equation \eqref{eq:lineqs}, arriving at the equivalent relation
\begin{equation}
\sum_{j=1}^{2n+1} (-1)^{j-1} \binom{6n}{m-4n+j}\alpha_{n,j} = \sum_{j=1}^{2n+1} \left(\binom{2n+j-2}{m-4n+j}+\binom{4n-j}{m-6n}\right)\alpha_{n,j}.
\label{eq:binomial-summation-family}
\end{equation}
It is now worth noting for aesthetic reasons, although not strictly necessary for our purposes, that, as occurred in the proof of Lemma~\ref{lem:normalization} above, a version of this identity also holds when $n$ is a half-integer, although to get the correct identity it is necessary to insert in two places a small sign correction of $(-1)^{2n}$. Making this correction, and substituting $\alpha_{n,k}$ into the equation, with the irrelevant scaling factor $\frac{(4n+1)!}{(6n+1)(2n)!^2}$ (which does not depend on $j$) removed from both sides, the identity we need to prove becomes
\begin{align}
\nonumber
&\sum_{j=1}^{2n+1} (-1)^{2n+j} \frac{\binom{2n}{j-1}}{\binom{6n}{2n+j-1}}\binom{6n}{m-4n+j} 
\qquad\qquad\qquad \\ &
\hspace{30pt}
= \sum_{j=1}^{2n+1} \frac{\binom{2n}{j-1}}{\binom{6n}{2n+j-1}}\left(\binom{2n+j-2}{m-4n+j}+(-1)^{2n} \binom{4n-j}{m-6n}\right).
\label{eq:binomial-sum-symmetric1}
\end{align}
Now replace $n$ by $n/2$, and perform two additional simplifications of a purely cosmetic nature, namely shifting the summation index $j$ by $1$ and replacing the parameter $m$ by $m+n-1$. This brings \eqref{eq:binomial-sum-symmetric1} to the form
\begin{align}
&\sum_{j=0}^{n} (-1)^{n+j} \frac{\binom{n}{j}}{\binom{3n}{n+j}}\binom{3n}{m+j-n} 
\qquad\qquad\qquad \nonumber \\ &
\hspace{50pt}
= \sum_{j=0}^{n} \frac{\binom{n}{j}}{\binom{3n}{n+j}}
\left(\binom{n+j-1}{m+j-n}+(-1)^n \binom{2n-j-1}{m-2n-1}\right),
\label{eq:binomial-sum-symmetric2}
\end{align}
where $n$ is assumed to be a positive integer and $m$ ranges over the values $0,1,2,\ldots,4n$.

Our next step is to exploit a subtle symmetry of the family of identities \eqref{eq:binomial-sum-symmetric2} in order to obtain one final simplification. The key observation is that the binomial coefficient $\binom{3n}{m-4n+j}$ on the left-hand side is invariant under the substitution $j\to n-j,\ m\to 4n-m$. On the right-hand side, the same substitution transforms the two binomial coefficients $\binom{n+j-1}{m+j-n}$ and $\binom{2n-j-1}{m-2n-1}$ into each other. Note further that when $0\le m\le 2n-1$, only the first of those two coefficients can be nonzero (under the constraint that $0\le j\le n$), and when $2n+1\le m\le 4n$, only the second one can be nonzero. When $m=2n$ both of them are zero, and in this case it is easy to verify that the left-hand side is a sum with alternating signs of the binomial coefficients $\binom{n}{j}$ and therefore also zero. Thus, by symmetry it is enough to verify \eqref{eq:binomial-sum-symmetric2} for $0\le m\le 2n-1$, and for that range of values it becomes the relation
\begin{equation}
\sum_{j=0}^{n} (-1)^{n+j} \frac{\binom{n}{j}\binom{3n}{m+j-n} }{\binom{3n}{n+j}}
= \sum_{j=0}^{n} \frac{\binom{n}{j}\binom{n+j-1}{m+j-n}}{\binom{3n}{n+j}}.
\label{eq:binomial-sum-asymmetric}
\end{equation}
Following this preparation, the well-known algorithmic methods for proving hypergeometric summation identities---specifically, Zeilberger's algorithm---once again come to our aid. Denote
\begin{align*}
F_1(n,m,j) &= (-1)^{n+j} \frac{\binom{n}{j}\binom{3n}{m+j-n} }{\binom{3n}{n+j}},  \qquad S_1(n,m)=\sum_{j=0}^n F_1(n,m,j),\\
F_2(n,m,j) &= \frac{\binom{n}{j}\binom{n+j-1}{m+j-n}}{\binom{3n}{n+j}}, \phantom{(-1)^{n+j}} \qquad S_2(n,m)=\sum_{j=0}^n F_2(n,m,j), \end{align*}
and define the additional auxiliary functions
\begin{align*}
U(n,m) &= (m+2)(2n-m-1), \\
V(n,m) &= -2m^2-5n^2+8mn+9n-4m-2, \\
W(n,m) &= (m-4n)(2n-m-1), \\
R_1(n,m,j) &= \frac{j(3n+1)(2n-j+1)}{n-m-j-1}, \\
R_2(n,m,j) &= \frac{j(m-2n+1)(2n-j+1)}{n-m-j-1}, \\
G_1(n,m,j) &= R_1(n,m,j)F_1(n,m,j), \\
G_2(n,m,j) &= R_2(n,m,j)F_2(n,m,j).
\end{align*}
The functions $U(n,m)$, $V(n,m)$, $W(n,m)$, $R_1(n,m,j)$, and $R_2(n,m,j)$  were found using Zeilberger's algorithm, and one may verify by a routine calculation (see the companion package \cite{romik-mathematica}) that they satisfy the pair of algebraic identities
\begin{align}
U(n,m)& F_\alpha(n,m+2,j)  + V(n,m)F_\alpha(n,m+1,j) + W(n,m)F_\alpha(n,m,j) 
\nonumber \\ &  = 
G_\alpha(n,m,j+1)-G_\alpha(n,m,j) \qquad\qquad (\alpha=1,2),
\label{eq:creative-telescoping}
\end{align}
where, importantly, the polynomials $U(n,m), V(n,m), W(n,m)$ do not depend on $j$ or $\alpha$. Summing both sides of \eqref{eq:creative-telescoping} over all integer $j$ results in a telescoping sum on the right-hand side, which (since $G_\alpha(n,m,j)=0$ if $j<0$ or $j>n$) is equal to $0$. Thus, we get the relations
\begin{equation} \label{eq:second-order-recurrence}
U(n,m)S_\alpha(n,m+2) + V(n,m)S_\alpha(n,m+1) + W(n,m)S_\alpha(n,m) = 0
\end{equation}
for $\alpha=1,2$. That is, we have shown that the sums $S_1(n,m), S_2(n,m)$ satisfy the same second-order linear recurrence relation with respect to the variable~$m$.

Finally, note that $U(n,m)\neq 0$ for $0\le m\le 2n-2$, so for those values of $m$ we can solve the recurrence \eqref{eq:second-order-recurrence} for $S_\alpha(n,m+2)$, expressing it in terms of $S_\alpha(n,m+1)$ and $S_\alpha(n,m)$. We therefore see that it is enough to verify that $S_1(n,0)=S_2(n,0)$ and $S_1(n,1)=S_2(n,1)$, after which the general equality $S_1(n,m)=S_2(n,m)$ (which is precisely the relation \eqref{eq:binomial-sum-asymmetric} we are trying to prove) follows for all $0\le m\le 2n$ by induction on $m$. The verification of the base cases $m=0,1$ is trivial and is left as an exercise.
\end{proof}

Combining Lemmas~\ref{lem:system-of-equations}, \ref{lem:normalization} and \ref{lem:lineqs} establishes the summation identity \eqref{eq:identity-betank} with the coefficients $\beta_{n,k}$ as defined in \eqref{eq:coeffs-guess2}, and therefore finishes the proof of \eqref{eq:bernoulli}.

\section{Eisenstein series}

\label{sec:eisenstein}

The key idea in our proof of the Bernoulli number summation identity \eqref{eq:bernoulli} in the last section was to find a pair of functions $f_n(p,q)$, $g_n(p,q)=f_n(p,q)-f_n(p+q,q)-f_n(p,p+q)$ which are both homogeneous polynomials in $p^{-1}, q^{-1}$ of degree $6n+2$, and such that $f_n(p,q)$ contains only a specific set of even powers of $p^{-1}$ and $q^{-1}$. Zagier, who introduced this method in \cite{zagier} to prove a more standard identity, already made the more general observation (see section 8 of his paper) that any summation identity proved using the same technique automatically ``lifts'' to a summation identity for the Eisenstein series, by a modification of the argument. To see how this works in our setting, instead of summing $g_n(p,q)$ over all pairs of positive integers $p,q$ as we did in \eqref{eq:sum-gofpq1}, \eqref{eq:sum-gofpq2}, now sum instead over all pairs of complex numbers $p,q$ ranging in the ``half-lattice''
$$ \Lambda_+(z) = \{ p=mz + n \,:\, m>0\textrm{ or }[m=0\textrm{ and }n>0] \} $$
(where $z$ is a fixed complex number in the upper half-plane $\mathbb{H}$, which will shortly become the argument of the Eisenstein series). This summation yields, on the one hand (analogously to \eqref{eq:sum-gofpq1}), the expression
$$
\sum_{k=1}^n \beta_{n,k} \left(\sum_{p,q \in \Lambda_+(z)}^\infty p^{-2n-2k}q^{-4n+2k-2} \right) = \sum_{k=1}^n \beta_{n,k} \frac{G_{2n+2k}(z)}{2}\frac{G_{4n-2k+2}(z)}{2}.
$$
On the other hand, for two elements $r=mz+n$ and $s=m'z+n'$ of the half-lattice $\Lambda_+(z)$, denote the order relation $r\succ s$ to mean that 
$m>m'$ or [$m=m'$ and $n>n'$], and note that $r\succ s$ if and only if $r=s+q$ for some $q\in\Lambda_+(z)$. Then the summation above can also be written, in analogy with~\eqref{eq:sum-gofpq2}, as
\begin{align*}
& \sum_{p,q \in \Lambda_+(z)}^\infty f_n(p,q) - \sum_{p,q \in \Lambda_+(z)}^\infty f_n(p+q,q) - \sum_{p,q \in \Lambda_+(z)}^\infty f_n(p,p+q) 
\nonumber \\ &= \Bigg(
\sum_{r,s \in \Lambda_+(z)}^\infty - \sum_{\begin{array}{c} \\[-17pt]\scriptstyle r,s \in \Lambda_+(z) \\[-5pt] \scriptstyle r \succ s \end{array}} -
\sum_{\begin{array}{c} \\[-17pt]\scriptstyle r,s \in \Lambda_+(z) \\[-5pt] \scriptstyle s \succ r \end{array}}
\Bigg) f_n(r,s)
= \sum_{r \in \Lambda_+(z)}^\infty f_n(r,r) \nonumber \\ &= \left(\sum_{j=1}^{2n+1} \alpha_{n,j} \right) \frac{G_{6n+2}(z)}{2} 
=\frac{G_{6n+2}(z)}{2}.
\end{align*}
Thus we obtain precisely the identity \eqref{eq:eisenstein}, proving Theorem~\ref{thm:eisenstein}.
\qed

\bigskip
In connection with the proof above, we remark that Zagier's observation about the lifting of summation relations to Eisenstein series was made in the context of considering general summation relations of the form 
\begin{equation}
\label{eq:eisenstein-linear-relation}
G_{2k} = \sum_{2\le j<k} c_j G_{2j} G_{2k-2j}
\end{equation}
expressing a given Eisenstein series in terms of the Eisenstein series with lower index. 
It is apparent from the formula for $\dim M_{2k}$ mentioned in Subsection~\ref{subsec:bernoulli} that for any fixed $k$ there must be a certain number of linearly independent relations of this kind (the number of such relations increases linearly with $k$), and one can use standard algorithms from linear algebra and the theory of modular forms to find all such relations.
From this point of view, the existence of summation relations of the form \eqref{eq:eisenstein-linear-relation}---again, for specific values of $k$---is not especially surprising.

What is \emph{not} necessarily apparent, and it is not clear to us from the discussion in \cite{zagier} whether it was envisioned by Zagier, was that his technique can also be used to prove interesting \emph{families} of summation relations (other than the single relatively trivial example discussed in section 1 of his paper) that hold for many values of $k$ and have the kind of elegant structure that our identity \eqref{eq:eisenstein} possesses. The discovery that this is in fact the case makes Zagier's technique appear a lot more exciting, and opens the way to discovering and proving---possibly in an algorithmic fashion---additional families of identities; see Section~\ref{sec:final-comments} for a further brief discussion of these possibilities.

\section{The function $\wzeta(s)$ and its analytic continuation}
\label{sec:wzeta}

The goal of this section is to prove Theorems~\ref{thm:analytic-continuation} and \ref{thm:wzeta-specialvalues}. Our approach is based to a large extent on an integral representation derived by \nobreak{Matsumoto} for the more general Dirichlet series $\zeta_{MT}(s_1,s_2,s_3)$ defined in \eqref{eq:matsumoto-mordell-tornheim} (see equation (5.3) in \cite{matsumoto}).

\bigskip\noindent
\textbf{The region of convergence: proof of Theorem~\ref{thm:analytic-continuation}(1).}
We start by giving a simple argument to prove the claim of Theorem~\ref{thm:analytic-continuation}(1) about the precise region of convergence of the Dirichlet series defining $\wzeta(s)$. Observe that
\begin{align*}
\wzeta(s) &=\sum_{j,k=1}^\infty \frac{1}{(j k(j+k))^s} \\ &= 
 \left( \sum_{j> k \ge 1}^\infty \frac{1}{(j k(j+k))^s} + 
\sum_{k>j \ge 1}^\infty \frac{1}{(j k(j+k))^s} + 
\sum_{k=j \ge 1}^\infty \right) \frac{1}{(j k(j+k))^s} 
\\ &=  2\sum_{j>k\ge 1}^\infty \frac{1}{j^{2s} k^s \left(1+\frac{k}{j}\right)^s} + 2^{-s} \zeta(3s).
\end{align*}
Here, the series for $2^{-s}\zeta(3s)$ converges if and only if $\re(s)>1/3$, and the summation over $j>k\ge1$ converges absolutely if and only if the simpler series
\begin{equation} \label{eq:simpler-series}
\sum_{j>k\ge 1}^\infty \frac{1}{\left|j^{2s} k^s\right|}
=\sum_{j>k\ge 1}^\infty \frac{1}{j^{2\sigma} k^\sigma}
\end{equation}
(where $\sigma=\re(s)$)
converges absolutely, since on the range of summation the multiplicative factor $\big(1+\frac{k}{j}\big)^{-s}$ is bounded between $2^{-s}$ and $1$. It is now trivial to check that the series in \eqref{eq:simpler-series} converges absolutely and uniformly on compacts (which ensures that the sum is a holomorphic function) if and only if $\re(s)>2/3$, for example by comparison with the integral
$$ \int_1^\infty \int_x^\infty \frac{1}{x^{2\sigma} y^\sigma} \,dy\,dx = \frac{1}{\sigma-1} \int_1^\infty x^{1-3\sigma}\, dx. 
\rlap{\qquad\qquad\qquad\qedsymbol}
$$

\paragraph{Analytic continuation: proof of Theorem~\ref{thm:analytic-continuation}(2).}
Next, we proceed with the analytic continuation of $\wzeta(s)$ to a meromorphic function on $\C$. We recall Matsumoto's ingenious method to achieve this (which he applied in a more general context), but add a few more details and motivation that are missing from his analysis and that we feel add some important insight in view of our newer discoveries about the significance of $\wzeta(s)$. Start with a standard contour integral formula from the class of Mellin-Barnes integrals \cite[eq.~(3.3.9), p.~91]{paris-kaminski}, namely
\begin{equation}
\label{eq:mellin-barnes}
\Gamma(s) (1+\lambda)^{-s} = \frac{1}{2\pi i} \cint_{(\alpha)} \Gamma(s+z)\Gamma(-z) \lambda^z\,dz,
\end{equation}
where ${(\alpha)}$ denotes the vertical contour from $\alpha-i\infty$ to $\alpha+i\infty$ (with $i=\sqrt{-1}$, here and throughout), and where we assume that 
$\lambda \in \C\setminus (-\infty,0]$ and $-\re(s)<\alpha<0$. Setting $\lambda=k/j$, multiplying by $j^{-2s}k^{-s}$ and summing over $j,k\in \N$ gives that
\begin{align}
\Gamma(s) \wzeta(s) &= 
\Gamma(s)  \sum_{j,k=1}^\infty j^{-2s} k^{-s}\left(1+\frac{k}{j}\right)^{-s}
\nonumber \\&= \frac{1}{2\pi i} \sum_{j,k=1}^\infty j^{-2s} k^{-s} \cint_{(\alpha)} \Gamma(s+z)\Gamma(-z) j^{-z} k^z \,dz
\nonumber \\ &=\frac{1}{2\pi i}\cint_{(\alpha)} \Gamma(s+z)\Gamma(-z)  \sum_{j,k=1}^\infty  j^{-2s-z} k^{-s+z}  \,dz
\nonumber \\ &=\frac{1}{2\pi i}\cint_{(\alpha)} \Gamma(s+z)\Gamma(-z)  \zeta(2s+z)\zeta(s-z)  \,dz,
\label{eq:wzeta-integral-rep1}
\end{align}
provided the infinite summations converge and that the interchange of the summation and integration operations can be justified. Note that the two zeta series within the integrand can be summed precisely when $\re(2s+z)>1$ and $\re(s-z)>1$, that is, when $1-2\sigma<\alpha<\sigma-1$ (where $\sigma=\re(s)$ as before). Together with the constraints $-\sigma<\alpha<0$, we see that $\alpha$ and $\sigma$ must satisfy
\begin{equation}
\label{eq:feasible-region}
\max(-\sigma, 1-2\sigma) < \alpha < \min(0,\sigma-1).
\end{equation}
As illustrated in Figure~\ref{fig:sing-diag}(a), a feasible value of $\alpha$ satisfying these constraints exists if and only if $\sigma>2/3$, which is consistent with our earlier observations on the region of convergence of the Dirichlet series defining $\wzeta(s)$. If $\sigma>2/3$, and a valid value of $\alpha$ is chosen (for example, the value $\alpha=-1/3$ works), the interchange of the summation and integration operations is easy to justify using standard facts about the exponential decay of the gamma function, and the (at most) polynomial growth of the zeta function, along vertical lines, namely that for $z=x+iy$, $|y|\ge1$, we have that
\begin{align}
\label{eq:rateofgrowth-bound-gamma}
|\Gamma(x+iy)| &= O\left(|y|^A e^{-\tfrac12 \pi |y|}\right) , \\
\label{eq:rateofgrowth-bound-zeta}
|\zeta(x+iy)| &=  O(|y|^A).
\end{align}
where $A$ is a constant that depends only on $x$, and where both $A$ and the constant implicit in the big-$O$ notation are uniform as $x$ ranges in compact subsets of $\R$; see \cite[Section~2.1.3]{paris-kaminski} and \cite[Section~6.4]{edwards} for the proofs of these relations.

Summarizing the above discussion, we conclude that the integral in \eqref{eq:wzeta-integral-rep1} is indeed a valid representation for $\Gamma(s)\wzeta(s)$. We remark also that the bounds \eqref{eq:rateofgrowth-bound-gamma} and \eqref{eq:rateofgrowth-bound-zeta} together with standard facts from complex analysis will also provide easy justifications for all subsequent manipulations of contour integrals in this paper (e.g., shifting the contour and differentiation under the integral). We assume the reader is familiar with such techniques and will mention these justifications only in passing without going into detail.

The next step is to analytically continue the integral to a larger region. This is done by shifting the contour of integration and taking into account the residues. Note that the integrand has poles (as a function of $z$ with $s$ fixed) at:
\begin{itemize}
\item $z=s-1$ because of the factor $\zeta(s-z)$;
\item $z=1-2s$ because of the factor $\zeta(2s+z)$;
\item $z=0,1,2,\ldots$ because of the factor $\Gamma(-z)$;
\item $z=-s,-s-1,-s-2,\ldots$ because of the factor $\Gamma(s+z)$.
\end{itemize}
This means that we should look for a value of $\alpha$ for the new integration contour $(\alpha)$ for which $\sigma=\re(s)$ is allowed to range over a wide range of positive and negative values without intersecting the lines in the $\sigma$-$\alpha$ plane defined by the equations
\begin{align}
\label{eq:singularity-line1} \alpha&=\sigma-1, \\ 
\label{eq:singularity-line2} \alpha&=1-2\sigma, \\
\label{eq:singularity-line3} \alpha&=n\phantom{\ -\sigma-} \qquad (n=1,2,\ldots), \\
\label{eq:singularity-line4} \alpha&=-\sigma-n \qquad (n=0,1,\ldots).
\end{align}
The situation can be visualized graphically in terms of the ``singularity diagram'' shown in Figure~\ref{fig:sing-diag}(b), which shows the pairs $(\sigma,\alpha)$ that cause the integral to pass through a singularity point.

\begin{figure}[h!]
\hspace{-10pt}
\begin{tabular}{ccc}
\scalebox{0.48}{\includegraphics{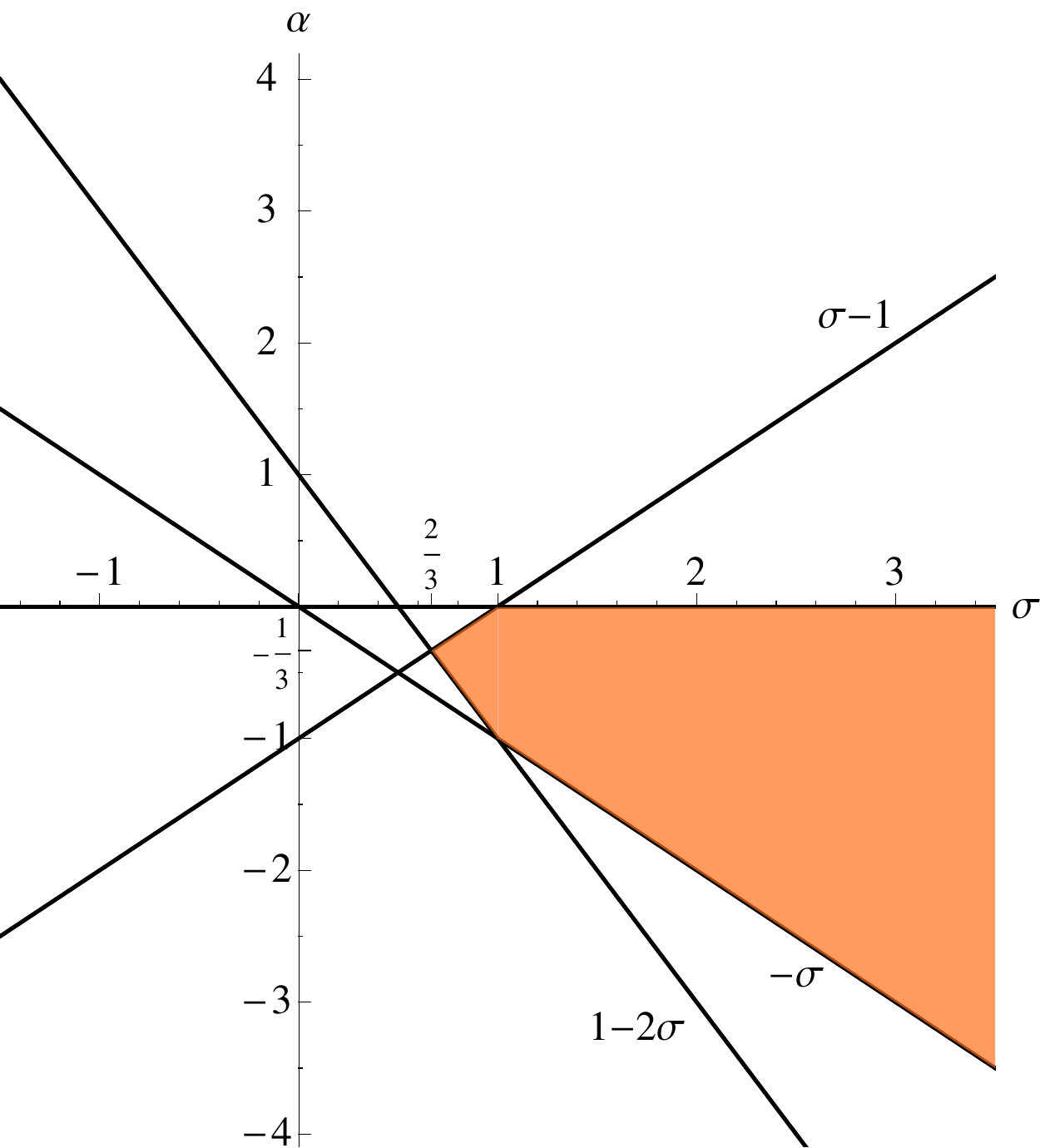}}
&  &
\scalebox{0.48}{\includegraphics{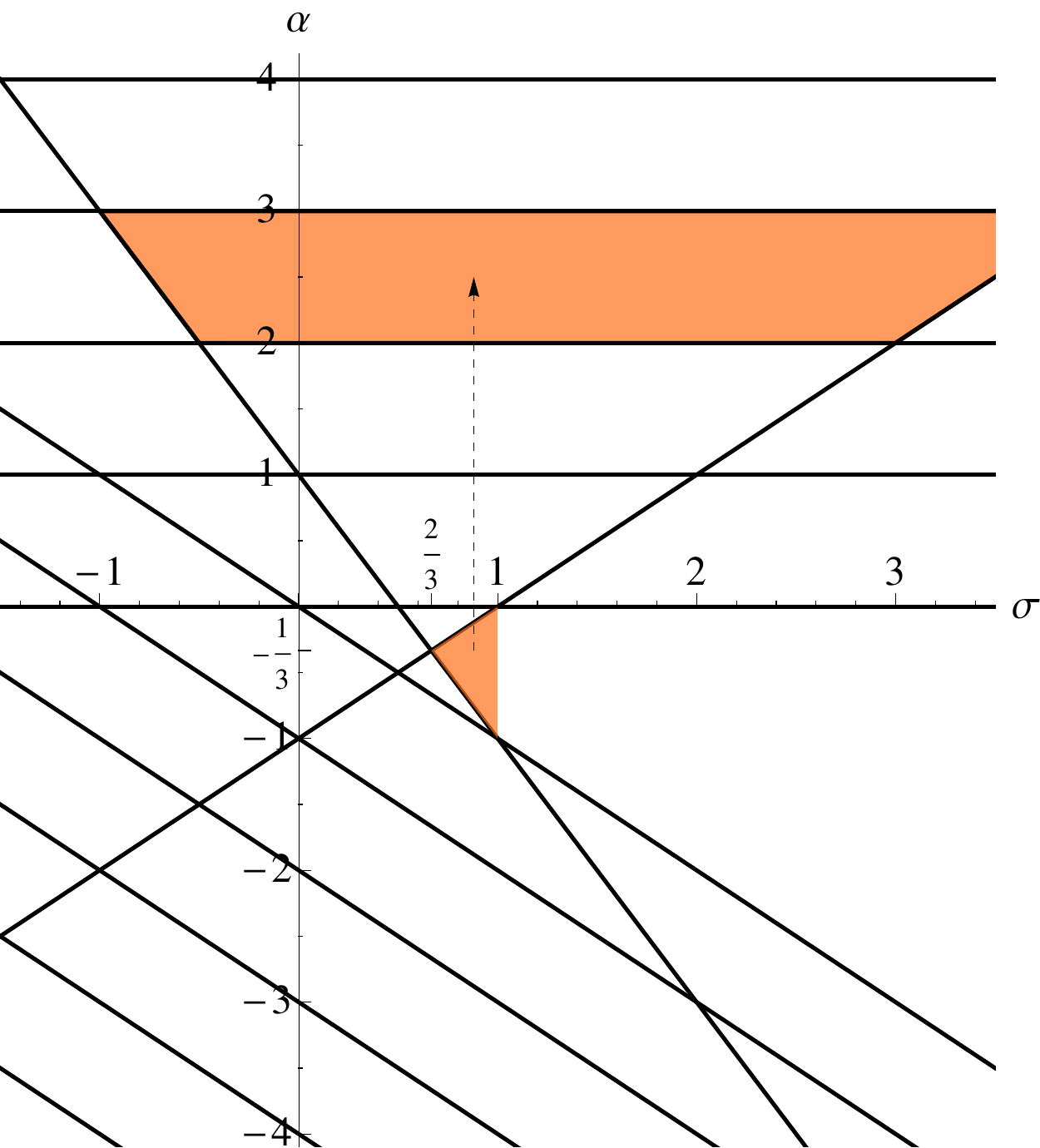}}
\\ (a) & & (b)
\end{tabular}
\begin{center}
\caption{(a) The shaded area in the diagram represents the feasible region of pairs $(\sigma,\alpha)$ for which the constraints \eqref{eq:feasible-region} are satisfied and therefore the
representation \eqref{eq:wzeta-integral-rep1} is valid.
(b) The lines in this diagram represent pairs of $(\sigma,\alpha)$ where the integrand in \eqref{eq:wzeta-integral-rep1} will run into a singularity. To analytically continue $\wzeta(s)$, we start at $(\sigma,-1/3)$ where $\sigma$ is initially allowed to vary in the interval $(\tfrac23,1)$, (the triangular shaded region) and move up along the $\alpha$ axis (which corresponds to shifting the contour of integration to the right) to a region of the diagram (the trapezoidal shaded area) where $\sigma$ can be varied across a much wider interval without running into singularities.
}
\label{fig:sing-diag}
\end{center}
\end{figure}

We are ready to shift the integration contour in the representation \eqref{eq:wzeta-integral-rep1}. Simplifying Matsumoto's argument slightly, we assume for convenience that initially $\tfrac23 < \sigma < 1$, so that we can avoid having to consider the double poles that result when two of the singularity lines intersect. The diagram in Figure~\ref{fig:sing-diag}(b) shows that it is a good idea to shift the contour $(\alpha)$ to the right (moving up in the coordinate system of the diagram) to $(\alpha')$ where $\alpha'=M-1/2$ and $M$ is some large integer. As the diagram shows, this involves skipping over the pole at $z=s-1$, which is a pole with residue $-\Gamma(2s-1)\Gamma(1-s)\zeta(3s-1)$; and the pole at $k$ for each $k=0,1,2,\ldots,M-1$, which has residue $\frac{(-1)^{k+1}}{k!}\Gamma(s+k)\zeta(2s+k)\zeta(s-k)$. Thus, by the residue theorem we have (still assuming $\tfrac23 < \sigma < 1$) that
\begin{align}
\nonumber \Gamma(s) \wzeta(s)
&= \Gamma(2s-1)\Gamma(1-s)\zeta(3s-1) 
\\ & \qquad + \sum_{k=0}^{M-1} 
\frac{(-1)^{k}}{k!}\Gamma(s+k)\zeta(2s+k)\zeta(s-k)
\nonumber \\ & \qquad  
+\frac{1}{2\pi i}\cint_{(M-1/2)} \Gamma(s+z)\Gamma(-z)  \zeta(2s+z)\zeta(s-z)  \,dz.
\label{eq:wzeta-integral-rep2}
\end{align}
Furthermore, Figure~\ref{fig:sing-diag}(b) shows graphically, and the equations \eqref{eq:singularity-line1}--\eqref{eq:singularity-line4} can easily be used to show algebraically,
that the integrand in \eqref{eq:wzeta-integral-rep2} encounters no singularities as $s$ ranges in the strip 
\begin{equation} \label{eq:strip}
\frac34-\frac{M}{2}<\sigma<M+\frac12.
\end{equation}
It follows (using \eqref{eq:rateofgrowth-bound-gamma}--\eqref{eq:rateofgrowth-bound-zeta}) that the integral defines a function which is holomorphic in the same strip. Since $M$ was arbitrarily large, we have therefore succeeded in extending $\wzeta(s)$ to a meromorphic function on the entire complex plane, proving Theorem~\ref{thm:analytic-continuation}(2).
\qed

\bigskip \noindent
\noindent \textbf{Analysis of the poles and values at nonnegative integers: proof of Theorem~\ref{thm:analytic-continuation}(3), Theorem~\ref{thm:wzeta-specialvalues}.}
We now extend Matsumoto's analysis further
by examining 
the singularities of $\wzeta(s)$ and its values at nonnegative integers. 
First, consider $\Gamma(s)\wzeta(s)$, whose analytic continuation is given by the right-hand side of \eqref{eq:wzeta-integral-rep2}. Here is a list of the poles of $\Gamma(s)\wzeta(s)$ contributed from each multiplicative factor in each of the additive terms, along with their residues.
\begin{enumerate}

\item The poles of the term $\Gamma(2s-1)\Gamma(1-s)\zeta(3s-1)$ are:
\begin{itemize}
\item A simple pole at $s=\tfrac23$ because of the factor $\zeta(3s-1)$. The residue is $\tfrac13 \Gamma\left(\tfrac13\right)^2$.
\item Potential simple poles at $s=\tfrac12, 0, -\tfrac12, -1,-\tfrac32,-2, \ldots$ because of the factor $\Gamma(2s-1)$. The residue at $s=\frac{1-k}{2}$ for each integer $k\ge0$ is $\frac{(-1)^k}{2\cdot k!} \Gamma\left(\frac{1+k}{2}\right)\zeta\left(\frac{1-3k}{2}\right)$. Note that for $k=3,7,11,15,\ldots$ the residue is $0$ because of the trivial zeros of $\zeta(s)$ at the negative even integers, so the poles of the term $\Gamma(2s-1)\Gamma(1-s)\zeta(3s-1)$ at these values are in fact removable singularities.
\item Simple poles at $s=1,2,3,\ldots$ because of the factor $\Gamma(1-s)$. The residue at $s=k+1$ is equal to $\frac{(-1)^{k+1}}{k!}\Gamma(2k+1)\zeta(3k+2) = \frac{(-1)^{k+1} (2k)!}{k!}\zeta(3k+2)$ for each integer $k\ge 0$.
\end{itemize}

\item The poles of $\frac{(-1)^k}{k!}\Gamma(s+k)\zeta(2s+k)\zeta(s-k)$ for $k=0,1,\ldots,M-1$ are:
\begin{itemize}
\item A simple pole at $s=k+1$ because of the factor $\zeta(s-k)$. The residue is 
$\frac{(-1)^k}{k!}\Gamma(2k+1)\zeta(3k+2) = \frac{(-1)^k(2k)!}{k!}\zeta(3k+2)$.
\item A simple pole at $s=\frac{1-k}{2}$ because of the factor $\zeta(2s+k)$. The residue is $\frac{(-1)^k}{2\cdot k!} \Gamma\left(\frac{1+k}{2}\right)\zeta\left(\frac{1-3k}{2}\right)$.
\item Potential simple poles at $s=-k,-k-1,-k-2,\ldots$ because of the factor $\Gamma(s+k)$. The residue at $s=-k-j$ for each integer $j\ge 0$ is $\frac{(-1)^k}{k!}\cdot \frac{(-1)^j}{j!} \zeta(-k-2j)\zeta(-2k-j)$. As above, some of them end up being removable singularities because of cancellation with a trivial zero of $\zeta(s)$.
\end{itemize}

\end{enumerate}
 
Adding up the contributions from the different terms (assuming that $M$ is arbitrarily large so we can ignore the edge effects), we see that the poles of $\Gamma(s)\wzeta(s)$ at $s=1,2,3,\ldots$ have residue $0$ and are in fact removable singularities (which could have been predicted in advance, as we already knew that $\Gamma(s)\wzeta(s)$ is analytic at those points). The remaining poles consist of a ``special'' pole at $s=\tfrac23$ with residue $\tfrac13\Gamma\left(\tfrac13\right)^2$, and a series of (potential) poles at half-integer points of the form $s=\frac{1-k}{2}$ for $k=0,1,2,\ldots$, where for an even value $k=2m$ the residue is
\begin{align*}
&\frac{(-1)^{2m}}{(2m)!}\Gamma\left(m+\tfrac12\right)\zeta\left(\frac{1-6m}{2}\right) 
\\ &= \frac{1\cdot3\cdot5\cdot\ldots\cdot (2m-1)}{(2m)! 2^m}\Gamma\left(\tfrac12\right)\zeta\left(\frac{1-6m}{2}\right)
= \frac{1}{4^m m!}\sqrt{\pi} \zeta\left(\frac{1-6m}{2}\right),
\end{align*}
and for $k=2m+1$ an odd integer, the residue is
\begin{align*}
\rho_m & :=\frac{(-1)^{2m+1}}{(2m+1)!}\Gamma(m+1)\zeta(-3m-1)
\\ & \qquad +\frac{(-1)^m}{m!}\sum_{j=0}^m \binom{m}{j}\zeta(-m-j)\zeta(-2m+j)
\\ & \phantom{:}= -\frac{m!}{(2m+1)!}\zeta(-3m-1) +
\frac{(-1)^m}{m!}\sum_{j=0}^m \binom{m}{j}\zeta(-m-j)\zeta(-2m+j).
\end{align*}
Note that $\rho_0=\zeta(0)^2-\zeta(-1)=\left(-\frac{1}{2}\right)^2+\frac{1}{12}=\frac13$. However, as we discovered to our surprise, the remaining values $\rho_1,\rho_2,\ldots$ are all zero! For odd values of $m$ this is trivially true since $\zeta(-3m-1)$ is a trivial zero of the Riemann zeta function, and in the summation over $j$, for each $j$ one of the factors (depending on the parity of $j$) is also a trivial zeta zero. For even values of $m$, denoting $m=2p$, we see that the claim reduces to the summation identity
$$
\zeta(-6p-1) = 
(4p+1)\binom{4p}{2p}\sum_{j=0}^{2p} \binom{2p}{j} \zeta(-2p-j)\zeta(-4p+j),
$$
which is precisely the version \eqref{eq:bernoulli-zeta-minus} of our Bernoulli number summation identity (note that the summation over $j$ can be taken only over odd values, again because of the trivial zeta zeros). Thus, \eqref{eq:bernoulli} and its reformulations \eqref{eq:bernoulli-zeta-plus}, \eqref{eq:bernoulli-zeta-minus} turn out to be equivalent to the claim that $\Gamma(s)\wzeta(s)$ has no poles at the negative integers. Since we, and Agoh and Dilcher before us, proved the former result using other methods, the latter claim about $\Gamma(s)\wzeta(s)$ follows.

We summarize our findings about $\Gamma(s)\wzeta(s)$ in the following theorem, which will play an important role in our asymptotic analysis of the representation enumeration function $r(n)$.

\begin{thm}
\label{thm:eofs-asym}
The meromorphic function
$\Gamma(s)\wzeta(s)$ possesses a simple pole at $s=\tfrac23$ with residue $\frac13 \Gamma\left(\tfrac13\right)^2$; a simple pole at $s=0$ with residue~$\frac13$; and for each $k=0,1,2,\ldots$ a simple pole at $s=\tfrac12-k$ with residue $ \frac{1}{4^k k!} \sqrt{\pi} \zeta\left(\frac{1-6k}{2}\right)$. It has no other singularity points.
\end{thm}

From Theorem~\ref{thm:eofs-asym} we can now easily infer part (3) of Theorem~\ref{thm:analytic-continuation}, and both claims \eqref{eq:wzeta-at-zero}, \eqref{eq:wzeta-trivialzeros} of Theorem~\ref{thm:wzeta-specialvalues}. This is based on two observations: first, 
for a complex number $s_0$ which is not a pole of the gamma function, we have
$$ \operatorname{Res}_{s=s_0}\Big(\wzeta(s)\Big) = \frac{1}{\Gamma(s_0)} \operatorname{Res}_{s=s_0}\Big(\Gamma(s)\wzeta(s)\Big).
$$
Second, at the poles $s_0=-n$ of $\Gamma(s)$, $n\ge 0$, we have
$$ 
\wzeta(s_0) = \frac{\operatorname{Res}_{s=s_0}\Big(\Gamma(s)\wzeta(s)\Big)}{\operatorname{Res}_{s=s_0}\Big(\Gamma(s)\Big)} = (-1)^n n! \operatorname{Res}_{s=-n}(\Gamma(s)\wzeta(s))
$$
(where the equation holds in the usual complex-analytic sense, that is, after removing the removable singularities).
Applying the first observation to $s_0=\tfrac23, \tfrac12,-\tfrac12,-\tfrac32,\ldots$, and the second observation to $s_0=0,-1,-2,\ldots$, gives the claims; for example, for $s_0=\tfrac23$ we get that
$$
\operatorname{Res}_{s=2/3} \Big(\wzeta(s)\Big) = \tfrac13\Gamma\left(\tfrac13\right)^2/\Gamma\left(\tfrac23\right),
$$
which is equal to $\frac{1}{2\pi \sqrt{3}} \Gamma\left(\tfrac13\right)^3$ by the reflection formula $\Gamma(x)\Gamma(1-x)=\frac{\pi}{\sin(\pi x)}$. We leave to the reader to fill in the details for the remaining poles of $\Gamma(s)\wzeta(s)$.

We have therefore proved parts (1)--(3) of Theorem~\ref{thm:analytic-continuation}; proved Theorem~\ref{thm:wzeta-specialvalues}; and established the additional interesting fact that \eqref{eq:wzeta-at-zero} implies, and is implied by, the summation identity \eqref{eq:bernoulli}. Furthermore, note that the evaluation $\wzeta(0)=\tfrac13$ is directly related to the fact that the identity \eqref{eq:bernoulli} does not hold for $n=0$. \qed

\bigskip \noindent
\textbf{Polynomial bound along vertical strips: proof of Theorem~\ref{thm:analytic-continuation}(4).}
To conclude our analysis of $\wzeta(s)$, it remains to prove part (4) of Theorem~\ref{thm:analytic-continuation}, namely the claim that \hbox{$|\wzeta(\sigma+it)|$} grows at most polynomially in $t$ as $|t|\to\infty$ with $\sigma$ ranging in a bounded interval $I=[a,b]$. This fact was mentioned briefly in \cite{matsumoto-bonner} (in the more general context of the multivariate Mordell-Tornheim zeta function) and is a fairly straightforward consequence of the representation \eqref{eq:wzeta-integral-rep2}: dividing that representation (with $M$ taken large enough so that the representation is valid for $\sigma\in I$) by $\Gamma(s)$ and using the standard relation $\Gamma(w+1)=w\Gamma(w)$, we get that

\begin{align}
\wzeta(s)
&= \frac{\Gamma(2s-1)\Gamma(1-s)\zeta(3s-1)}{\Gamma(s)}
\nonumber \\ & \qquad + \sum_{k=0}^{M-1} 
\frac{(-1)^{k}}{k!}s(s+1)\ldots (s+k-1)\zeta(2s+k)\zeta(s-k)
\nonumber \\ & \qquad  
+\frac{1}{2\pi i}\cint_{(M-1/2)} \frac{\Gamma(s+z)\Gamma(-z)  \zeta(2s+z)\zeta(s-z)}{\Gamma(s)}  \,dz.
\label{eq:wzeta-integral-rep3}
\end{align}
Here, all the summands except possibly the integral can be immediately seen using \eqref{eq:rateofgrowth-bound-gamma}--\eqref{eq:rateofgrowth-bound-zeta} to satisfy such a polynomial bound. For the integral, writing $s=\sigma+it$ and $z=x+iy$, we use \eqref{eq:rateofgrowth-bound-gamma}--\eqref{eq:rateofgrowth-bound-zeta} again to see that it can be bounded as
\begin{align}
\Bigg|\,\cint_{(M-1/2)} 
& \frac{\Gamma(s+z)\Gamma(-z)  \zeta(2s+z)\zeta(s-z)}{\Gamma(s)}  \,dz \Bigg|
\nonumber \\ &\le 
e^{\frac{\pi}{2} t} \int_{-\infty}^\infty e^{-\frac{\pi}{2}|t+y|}e^{-\frac{\pi}{2}|y|} \times \left|\textrm{Poly}(t,y)\right|\,dy,
\label{eq:polybound-convolution}
\end{align}
where $\operatorname{Poly}(t,y)$ denotes a polynomial factor (with powers and coefficients that are uniform as $\sigma$ ranges over compact intervals). However, a quick side calculation, which we omit, shows that for any integer $n\ge 0$ the function
$$
u\mapsto e^u \int_{-\infty}^\infty |v|^n e^{-|v|} e^{-|v+u|} \,dv
$$
is a polynomial in $u$. Applying this to the right-hand side of \eqref{eq:polybound-convolution} gives the desired polynomial bound, and finishes the proof of Theorem~\ref{thm:analytic-continuation}.\qed

\section{The formula for $\wzeta'(0)$: proof of Theorem~\ref{thm:wzeta-deriv0}}
\label{sec:wzeta2}

To prove the formula \eqref{eq:wzeta-deriv0} for $\wzeta'(0)$, start with the integral representation \eqref{eq:wzeta-integral-rep3}. In the case $M=2$, according to \eqref{eq:strip} the representation is valid when $-\tfrac14 < \re(s) < \tfrac52$. In particular, we have in a neighborhood of $s=0$ that
\begin{align}
\wzeta(s) &= \frac{\Gamma(2s-1)\Gamma(1-s)\zeta(3s-1)}{\Gamma(s)} + \zeta(s)\zeta(2s)
-s\zeta(s-1)\zeta(2s+1)
\nonumber
\\ & \qquad\qquad + \frac{1}{2\pi i}\cint_{(3/2)} \frac{\Gamma(-z)\Gamma(s+z)\zeta(2s+z)\zeta(s-z)}{\Gamma(s)}\,dz,
\label{eq:wzeta-rep-Mequals2}
\end{align}
where in the case of the removable singularity at $s=0$ the equality is understood in an appropriate limiting sense. Now inspect the Taylor expansion of each of the summands up to first order, using standard facts about the Taylor expansions of the gamma and zeta functions at $s=-1, 0, 1$. For the first summand, we have
\begin{align*}
& \frac{\Gamma(2s-1)\Gamma(1-s)\zeta(3s-1)}{\Gamma(s)} 
\\ &= 
\frac{(1+\gamma s +O(s^2))\left(-\frac{1}{2s}+(\gamma-1)+O(s)\right)\left(-\frac{1}{12}+3\zeta'(-1)s+O(s^2)\right)}{\frac{1}{s}-\gamma + O(s)}
\\ &=
\frac{1}{24} + \left( 
\frac{1}{12}-\frac32 \zeta'(-1)
\right)s + O(s^2).
\end{align*}
Similarly, for the second summand,
\begin{align*}
\zeta(s)\zeta(2s) &= \left(-\frac12-\frac12\log(2\pi)s+O(s^2)\right)
\left(-\frac12-\log(2\pi)s+O(s^2)\right)
\\ &=
\frac14 + \left(\frac34\log(2\pi)\right)s + O(s^2).
\end{align*}
and for the third summand,
\begin{align*}
-s\zeta(s-1)\zeta(2s+1) &= -s\left(\frac{1}{2s}+\gamma+O(s)\right)\left(-\frac{1}{12}+\zeta'(-1)s+O(s^2)\right)
\\ &= \frac{1}{24} + \left(\frac{1}{12}\gamma - \frac12 \zeta'(-1)\right)s + O(s^2).
\end{align*}
The coefficient of $s$ in the sum of the above three expansions is
\begin{equation} \label{eq:leadingupto}
\frac{1}{12}+\frac{1}{12}\gamma +\frac{3}{4} \log(2\pi) -2\zeta'(-1),
\end{equation}
which accounts for all the terms on the right-hand side of \eqref{eq:wzeta-deriv0} except for the integral. 

It remains to consider the integral in \eqref{eq:wzeta-rep-Mequals2}. Define 
$$
\Delta(s,z) = 
\frac{\Gamma(-z)\Gamma(s+z)\zeta(2s+z)\zeta(s-z)}{\Gamma(s)}.
$$
Differentiating $\Delta(s,z)$ with respect to $s$ and using the Taylor expansions 
\begin{align*}
\Gamma(s)&= \phantom{-}\frac{1}{s}-\gamma + O(s), \\[3pt]
\frac{\Gamma'(s)}{\Gamma(s)} &= -\frac{1}{s} + \gamma + O(s),
\end{align*}
a short computation, which we omit, shows that
$$
\frac{\partial}{\partial s}_{\raisebox{1pt}{\textrm{\large $|$}}s=0} \Delta(s,z) = 
\Gamma(z)\Gamma(-z)\zeta(z)\zeta(-z),
$$
which, using the functional equations $\Gamma(w+1)=w \Gamma(w)$ and $\Gamma(z)\Gamma(1-z)=\pi/\sin(\pi z)$, simplifies further to
$$
\frac{\partial}{\partial s}_{\raisebox{1pt}{\textrm{\large $|$}}s=0} \Delta(s,z)
= -\frac{\pi}{z\sin (\pi z)} \zeta(z)\zeta(-z).
$$
Thus, differentiating the integral in \eqref{eq:wzeta-rep-Mequals2} and setting $s=0$ gives 
\begin{align}
&\frac{\partial}{\partial s}_{\raisebox{1pt}{\textrm{\large $|$}}s=0} \left[
\frac{1}{2\pi i}\cint_{(3/2)} \frac{\Gamma(-z)\Gamma(s+z)\zeta(2s+z)\zeta(s-z)}{\Gamma(s)}\,dz
\right]
\nonumber \\[3pt] & 
=
-\frac1{2i} \cint_{(3/2)} \frac{\zeta(z)\zeta(-z)}{z \sin(\pi z)}\,dz
=
\frac12 \int_{-\infty}^\infty \frac{\zeta\left(\tfrac32+it\right)\zeta\left(-\tfrac32-it\right)}{\left(\tfrac32+it\right) \cosh(\pi t)}\,dt,
\label{eq:diff-ints0}
\end{align}
since $\sin\!\left(\pi\left(\tfrac32+it\right)\right)$ simplifies to $-\cosh(\pi t)$. Combining \eqref{eq:diff-ints0} 
with the computations leading up to \eqref{eq:leadingupto} finishes the proof of Theorem~\ref{thm:wzeta-deriv0}. \qed

\section{Asymptotics of generating functions for $SU(3)$ representations}

\label{sec:asym-genfuns}

In this section we begin applying our results on the function $\wzeta(s)$ to derive asymptotic results for the enumeration of representations of $SU(3)$. Let us start with a few general remarks about the broader context for our analysis. Theorem~\ref{thm:su3-asym} is formulated most naturally as an asymptotic enumeration result for representations of $SU(3)$, but one can think of the sequence $r(n)$ as enumerating a certain combinatorial class of integer partitions, namely the class of representations of an integer $n$ as a sum of integers of the form $a_{j,k}=\tfrac12 j k(j+k)$, where repetitions are allowed; the order of the summands is unimportant; and if $a_{j,k}=a_{j',k'}$ but $(j,k)\neq(j',k')$, the numbers $a_{j,k}$ and $a_{j',k'}$ are considered distinct. Thus, if one ignores the representation-theoretic aspect, our result broadly falls into the area of partition asymptotics. This topic, which had its beginnings in 1918 with the asymptotic expansion of Hardy and Ramanujan \cite{hardy-ramanujan} for the partition-counting function $p(n)$, has important connections to analytic number theory and the theory of modular forms (see \cite{apostol}). Highlights of the theory include Rademacher's refinement of Hardy and Ramanujan's formula into a convergent series for $p(n)$ \cite{rademacher} and Wright's asymptotic formula \cite{wright} for the number of plane partitions of $n$. See \cite{andrews-mocktheta}, \cite{barany},
\cite{bringmann-mahlburg},  \cite{holroyd-etal} for more recent results.

One landmark result in the theory was the theorem of Meinardus, who in his 1954 paper \cite{meinardus} identified a general strategy for proving asymptotic formulas for a wide class of combinatorial enumeration sequences $c(n)$ whose generating functions have an Euler product formula of the form
\begin{equation}
\label{eq:meinardus-genfun}
1+\sum_{n=1}^\infty c(n)x^n = \prod_{m=1}^\infty \frac{1}{(1-x^m)^{a_m}}
\end{equation}
for some sequence $(a_m)_{m=1}^\infty$ of nonnegative integers. Assuming certain technical conditions are satisfied, Meinardus's result states that $c(n)$ has an asymptotic formula of the form
\begin{equation} \label{eq:meinardus-asym} c(n) = (1+o(1))Cn^{-\gamma} \exp\left(D n^\delta\right)
\end{equation}
for certain explicitly computable constants $C, D, \gamma, \delta$. His result includes many earlier results (including \eqref{eq:hardy-ramanujan} and Wright's formula for plane partitions mentioned above) as special cases that can be deduced from it without much difficulty; see \cite[Chapter~6]{andrews} for examples and an accessible exposition of \nobreak{Meinardus's} result.

Now note that the generating function \eqref{eq:rofn-genfun} of our sequence $r(n)$ can in fact be written in the form \eqref{eq:meinardus-genfun}, with the exponents $a_m$ being given by $a_m = \#\{(j,k)\in\N^2\,:\,m=jk(j+k)/2\}$ (that is, the number of inequivalent $m$-dimensional irreducible representations of $SU(3)$). One might hope that Meinardus's theorem will apply in this case, but it turns out that it does not, and indeed, the asymptotic formula \eqref{eq:su3-asym} ends up having a more elaborate form than \eqref{eq:meinardus-asym}. Nonetheless, the general strategy employed by Meinardus and earlier authors for proving asymptotic formulas such as \eqref{eq:hardy-ramanujan} and \eqref{eq:meinardus-asym} (described for example in \cite[Section~VIII.6]{flajolet-sedgewick}) remains valid, and we were able to adapt it to our needs, although a few nontrivial technical hurdles need to be overcome. One conceptual innovation that simplifies the analysis somewhat is to use a probabilistic representation similar to Fristedt's probabilistic model for random integer partitions \cite{fristedt}.

The first step in the analysis, which we undertake in this section, is to understand the asymptotic behavior of the generating function of $r(n)$ and a few related functions, as the argument approaches a singularity point. The main tool we will use is the Mellin transform and its inverse. Define
\begin{align} 
G(x) &= \prod_{j,k=1}^\infty \frac{1}{1-x^{jk(j+k)/2}} = \sum_{n=0}^\infty r(n)x^n & (|x|<1),
\label{eq:Gofx-def} \\
f(t) &= \sum_{j,k=1}^\infty \exp\Big(-\tfrac12 jk(j+k) t\Big) & (t>0), \label{eq:foft-def} \\
h(t) &= \log G(e^{-t}), & (t>0).
\label{eq:hoft-def}
\end{align}
Note that the trivial bound 
$$\sum_{j,k=1}^\infty |x|^{jk(j+k)/2} \le \sum_{j,k=1}^\infty |x|^{(j+k)/2}
= \Bigg(\sum_{j=1}^\infty |x|^{j/2}\Bigg)^2
$$ 
shows that the product defining $G(x)$ (hence also the series) is absolutely convergent for complex $x$ satisfying $|x|<1$ and defines an analytic function, and that $f(t), h(t)$ are defined and finite for $t>0$. We will also need a few additional easy bounds that are given in the following lemma.

\begin{lem} 
\label{lem:hoft-asymbounds}
The functions $f(t)$ and $h(t)$ satisfy the asymptotic bounds
$$
f(t) = \begin{cases}
O(t^{-2/3}) & \textrm{as }t\searrow0, \\
O(e^{-t}) & \textrm{as }t\to \infty,
\end{cases}
\qquad
h(t) = \begin{cases}
O(t^{-2/3}) & \textrm{as }t\searrow0, \\
O(e^{-t}) & \textrm{as }t\to \infty.
\end{cases}
$$
\end{lem}

\begin{proof} Consider first $f(t)$; the bound $f(t)=O(e^{-t})$ as $t\to\infty$ is trivial and left as an exercise. The behavior of $f(t)$ as $t\searrow 0$ can be understood by writing
$$
f(t) = t^{-2/3} \sum_{j,k=1}^\infty \exp\left(-\tfrac12 (t^{1/3}j)(t^{1/3}k)(t^{1/3}j+t^{1/3}k\right) \, t^{1/3} \, t^{1/3},
$$
and noting that this is $t^{-2/3}$ times a Riemann sum (with $\Delta x = \Delta y = t^{1/3}$) for the double integral
$I=\int_0^\infty \int_0^\infty \exp\left(-\tfrac12 xy(x+y)\right)dx\,dy$. This integral can be seen to be finite by making the two-dimensional change of variables $u=x, v=\tfrac12xy(x+y)$, which, after a short computation that we omit, gives
$$
I= 2\int_0^\infty \int_0^\infty e^{-v}\frac{1}{\sqrt{u^4+8vu}}\,du\,dv.
$$
Making another substitution, namely $w=(8v)^{-1/3}u$ for the integral with respect to $u$, transforms this into
\begin{align*}
2\int_0^\infty  e^{-v} &\left(\int_0^\infty \frac{(8v)^{1/3}  dw}{\sqrt{(8v)^{4/3}w^4+(8v)^{4/3}w}}\right)\,dv
\\ &= 2\times 8^{-1/3} \int_0^\infty  v^{-1/3} e^{-v} dv\cdot \int_0^\infty \frac{dw}{\sqrt{w^4+w}} < \infty,
\end{align*}
proving therefore that $f(t)=O(t^{-2/3})$ (in fact $f(t)=(I+o(1))t^{-2/3}$) as $t\searrow 0$. Finally, for $h(t)$, note that
\begin{align}
h(t) &= -\sum_{j,k=1}^\infty \log \Big(1-\exp\left(-\tfrac12 j k (j+k)t\right)\Big)
\nonumber 
\\ &= \sum_{j,k=1}^\infty \sum_{m=1}^\infty \frac{1}{m}\exp\left(-\tfrac12 j k (j+k)mt\right)
\nonumber 
\\ &= \sum_{m=1}^\infty \frac{1}{m} \left(
\sum_{j,k=1}^\infty \exp\left(-\tfrac12 j k (j+k)mt\right)
\right)
= \sum_{m=1}^\infty \frac{1}{m} f(mt).
\label{eq:hoft-harmonic-sum}
\end{align}
The claim about the behavior $h(t)$ as $t\to\infty$ is again trivial, and for $t$ near $0$ we have
\begin{align*} 
|h(t)| &\le \sum_{m=1}^\infty \frac{1}{m} |f(mt)|
= O\left(\sum_{m=1}^\infty \frac{1}{m} (mt)^{-2/3} \right)
\\ &= O\left(\sum_{m=1}^\infty \frac{1}{m^{5/3}}\cdot t^{-2/3} \right) = O(t^{-2/3}).
\end{align*}
(Note that the $O(t^{-2/3})$ bound for $f(t)$ actually holds for all $t>0$, since $e^{-t}$ decays faster than $t^{-2/3}$ as $t\to\infty$.)
\end{proof}

We are now ready to prove a much stronger result about the asymptotic behavior of $h(t)$ and its derivatives as $t\searrow 0$.

\begin{thm}
\label{thm:hoft-asym}
As $t \searrow 0$, the function $h(t)$ has the asymptotic expansion
\begin{equation}
\label{eq:hoft-asym}
h(t) = \mu_1 t^{-2/3} + \mu_2 t^{-1/2} - \tfrac13 \log t + \wzeta'(0)+\tfrac13 \log 2 + O(t^{1/2}), 
\end{equation}
where $\mu_1, \mu_2$ are constants defined by
\begin{align*}
\mu_1 &= \frac{2^{2/3}}{3}\Gamma\left(\tfrac13\right)^2 \zeta\left(\tfrac53\right), \\
\mu_2 &= \sqrt{2 \pi} \zeta\left(\tfrac12\right) \zeta\left(\tfrac32\right).
\end{align*}
Furthermore, the expansion \eqref{eq:hoft-asym} can be differentiated termwise; specifically, the first two derivatives $h'(t)$ and $h''(t)$ have the asymptotic expansions
\begin{align}
\label{eq:hoft-asym-deriv}
h'(t) &= -\tfrac23 \mu_1 t^{-5/3}\  -\tfrac12 \mu_2 t^{-3/2} - \tfrac13 t^{-1}+ O(t^{-1/2}),
\\
h''(t) &= \phantom{-}\tfrac{10}{9} \mu_1 t^{-8/3} +\tfrac34 \mu_2 t^{-5/2} + \tfrac13 t^{-2} + O(t^{-3/2}).
\label{eq:hoft-asym-deriv2}
\end{align}
\end{thm}

\begin{proof}
First, we wish to show that $f(t)$, and therefore also $h(t)$, can be related to the analytic function $\wzeta(s)$ via the Mellin transform. Start by considering $f(t)$. Computing its Mellin transform, which we denote by $\mellin{f(t)}(s)$, we have that
\begin{align}
\nonumber \mellin{f(t)}(s) &:= \int_0^\infty f(t) t^{s-1}\,dt = \sum_{j,k=1}^\infty \int_0^\infty \exp\Big(-\tfrac12 jk(j+k) t\Big) t^{s-1}\,dt 
\\ &\phantom{:}= \sum_{j,k=1}^\infty 
\left(\tfrac12jk(j+k)\right)^{-s} \,\Gamma(s) = 2^s \Gamma(s)\wzeta(s). 
\label{eq:foft-mellin}
\end{align}
Next, consider $h(t)$. 
Taking the Mellin transform of \eqref{eq:hoft-harmonic-sum}, we obtain easily, using the standard linearity and scaling properties of the Mellin transforms, that
\begin{align}
\nonumber \mellin{h(t)}(s) &= \sum_{m=1}^\infty \frac{1}{m} \mellin{f(mt)}(s) = 
\sum_{m=1}^\infty \frac{1}{m} m^{-s} \mellin{f(t)}(s) \\ &= 
2^s \Gamma(s)\wzeta(s)\zeta(s+1).
\label{eq:hoft-mellin}
\end{align}
By Lemma~\ref{lem:hoft-asymbounds}, both \eqref{eq:foft-mellin} and \eqref{eq:hoft-mellin} are valid for all complex $s$ satisfying $\re(s)>2/3$. (This is important for the proof below but also provides a useful consistency check with Theorem~\ref{thm:analytic-continuation}(1).)

We are now in a good position to apply the method of Mellin transform asymptotics, described for example in \cite[Appendix~B.7]{flajolet-sedgewick}. The main idea is that the asymptotic behavior of $h(t)$ as $t\searrow 0$ is closely tied to the singularity structure in the complex plane of its Mellin transform. Specifically, we apply the Mellin inversion formula \cite[eq.~(3.1.5), p.~80]{paris-kaminski} to deduce that $h(t)$ has the contour integral representation
\begin{equation} \label{eq:hoft-invmellin}
h(t) = \frac{1}{2\pi i} \cint_{(\alpha)} 2^s \Gamma(s)\wzeta(s)\zeta(s+1) t^{-s}\,ds,
\end{equation}
where $\alpha>2/3$, and $(\alpha)$ denotes as before the vertical contour from $\alpha-i \infty$ to $\alpha+i \infty$.
Shifting the contour of integration to the left past a few of the poles of the integrand will now produce the main asymptotic contributions from the residues, and the integral will become an error term whose magnitude is easily controlled. The relevant poles past which we will move the contour are at $s=2/3, 1/2, 0, -1/2$. The first two are simple poles, with respective residues
\begin{align*}
\operatorname{Res}_{s=2/3} \Big[
2^s \Gamma(s)\wzeta(s)\zeta(s+1) t^{-s}\Big] = \frac{2^{2/3}}{3}\Gamma\left(\tfrac13\right)^2\zeta\left(\tfrac53\right)t^{-s} = \mu_1 t^{-s}, \\
\operatorname{Res}_{s=1/2} \Big[
2^s \Gamma(s)\wzeta(s)\zeta(s+1) t^{-s}\Big] = \sqrt{2\pi}\zeta\left(\tfrac12\right)\zeta\left(\tfrac32\right)t^{-s} = \mu_2 t^{-s}.
\end{align*}
The pole at $s=0$ is a double pole since each of the factors $\Gamma(s)$ and $\zeta(s+1)$ contributes a first-order singularity. By the well-known facts that $\Gamma(s)=1/s-\gamma+O(s)$ and $\zeta(s+1)=1/s+\gamma+O(s)$ near $s=0$ (where $\gamma$ is the Euler-Mascheroni constant), the Laurent expansion of $2^s \Gamma(s)\wzeta(s)\zeta(s+1) t^{-s}$ around $s=0$ is easily computed as
\begin{align*}
\Gamma(s)\wzeta(s)\zeta(s+1) (t/2)^{-s} &= \left(\frac{1}{s}-\gamma + O(s)\right)
\left(\frac{1}{s}+\gamma+O(s)\right)
\\ & \ \quad \times\left(\tfrac13 + \wzeta'(0)s + O(s^2) \right)(1 - s\log (t/2) + O(s^2)) 
\\ & = \frac{1}{3s^2} +\left(\wzeta'(0)+\tfrac13\log 2 - \tfrac13 \log t\right) \frac{1}{s} + O(1).
\end{align*}
Thus, the residue at $s=0$ is $\wzeta'(0)+\tfrac13\log 2 - \tfrac13 \log t$. Finally, the residue at $s=-1/2$ is $\big(\tfrac14 \sqrt{\pi/2} \zeta(-5/2) \zeta(1/2)\big) t^{1/2} =: \nu t^{1/2}$.

Now shift the contour in \eqref{eq:hoft-invmellin} past the four poles, say to $(\alpha')$ where $\alpha'=-1$. By the residue theorem (with easy justification provided by \eqref{eq:rateofgrowth-bound-gamma}, \eqref{eq:rateofgrowth-bound-zeta} and \eqref{eq:wzeta-polybound}), we get that
\begin{align}
h(t) &= \mu_1 t^{-2/3} + \mu_2 t^{-1/2} - \tfrac13 \log t + \wzeta'(0)+\tfrac13 \log 2 + \nu t^{1/2} 
\nonumber \\ & \qquad + \frac{1}{2\pi i} \cint_{(-1)} 2^s \Gamma(s) \wzeta(s) \zeta(s+1) t^{-s}\,ds.
\label{eq:hoft-invmellin2}
\end{align}
It is easy to see that the integral is $O(t^{-1})$, so this proves \eqref{eq:hoft-asym}. Finally, \eqref{eq:hoft-asym-deriv}--\eqref{eq:hoft-asym-deriv2} follow by differentiating both sides of \eqref{eq:hoft-invmellin2}.
\end{proof}

\section{A probabilistic model for representations of $SU(3)$}

Next, we take advantage of the product structure of the generating function \eqref{eq:rofn-genfun} to interpret the coefficients $r(n)$ in terms of a probabilistic model involving a family of independent random variables distributed according to the geometric distribution. Such probabilistic models are widely applied to the asymptotic enumeration of combinatorial structures and to the analysis of their probabilistic properties (see \cite{arratia-tavare-barbour}). A version specifically tailored to the study of integer partitions was introduced by Fristedt \cite{fristedt} (and is sometimes referred to as Fristedt's conditioning device) and forms the prototype for a variety of similar models, considered, e.g., in \cite{vershik}.

Let $t>0$ denote a parameter. For integers $j,k\ge1$, let $X_{j,k}$ denote a random variable defined on some probability space having the distribution $X \sim \operatorname{Geom}_0(1-e^{-jk(j+k)t/2})$, namely the geometric distribution (in the version that starts at $0$, hence the notation $\operatorname{Geom}_0(\cdot))$ with parameter $p_{j,k}=
1-e^{-jk(j+k)t/2}$. That is, we have
$$ \prob_t( X_{j,k} = m ) = p_{j,k}(1-p_{j,k})^m = (1-e^{-jk(j+k)t/2}) e^{-jk(j+k)mt/2} $$
for $m\ge 0$,
where the notation $\prob_t(\cdot)$ denotes the probability with respect to the parameter value~$t$. We assume that all the variables $(X_{j,k})_{j,k\ge1}$ are simultaneously defined on the same probability space and form an independent family of random variables.

We interpret the variables $(X_{j,k})_{j,k\ge0}$ as encoding the structure of a certain random representation of $SU(3)$, namely
\begin{equation} \label{eq:random-rep} W = \bigoplus_{j,k=1}^\infty X_{j,k} W_{j,k}
\end{equation}
which consists of a sum of $X_{j,k}$ copies of each of the irreducible representations $W_{j,k}$ (see Subsection~\ref{subsec:asym-rep}). The total dimension of $W$ is therefore the random variable, which we denote $N$, given by
\begin{equation} \label{eq:def-N}
N = \sum_{j,k=1}^\infty \tfrac12 jk(j+k) X_{j,k}.
\end{equation}
It is easily seen from the Borel-Cantelli lemma from probability theory that with probability $1$ only finitely many terms in the sum are nonzero, and therefore that $N$ is almost surely finite.

If $(m_{j,k})_{j,k=1}^\infty$ is an array of (nonrandom) nonnegative integers, we have
\begin{align}
\nonumber
\prob_t&\left( \bigcap_{j,k=1}^\infty \left\{ X_{j,k}=m_{j,k} \right\} \right) = \prod_{j,k=1}^\infty \left[ \left(1-e^{-jk(j+k)t/2}\right)e^{-jk(j+k)m_{j,k}t/2} \right]
\\ &= \frac{e^{-n t}}{G(e^{-t})}, 
\label{eq:prob-specific-rep}
\end{align}
where $G(x)$ is defined in \eqref{eq:Gofx-def}, and $n=\sum_{j,k=1}^\infty \tfrac12 jk(j+k) m_{j,k}$. This can be interpreted as the probability in our random model of the event that the random representation $W$ in \eqref{eq:random-rep} is equal to the specific nonrandom representation $\oplus_{j,k=1}^\infty m_{j,k}W_{j,k}$, which has dimension $n$. Note that representations with the same dimension have equal probabilities of being observed under the measure $\prob_t(\cdot)$. Summing the probabilities \eqref{eq:prob-specific-rep} over all $r(n)$ representations of dimension $n$ therefore gives that
\begin{equation}
\label{eq:prob-dim-n}
\prob_t(N=n) = \frac{e^{-nt}r(n)}{G(e^{-t})}.
\end{equation}
This key formula will be of crucial significance in our analysis; it relates our integer sequence $r(n)$ to the distribution of the random variable $N$, and will therefore allow us to translate results about that distribution into information about $r(n)$.

Let $\expec_t\left(\cdot\right)$ and $\var_t(\cdot)$ denote the expectation and variance, respectively, relative to the parameter value $t$. The following lemma provides another connection of the random variable $N$ to a function we have been studying.

\begin{lem}
\label{lem:formulas-expecvar}
The expectation and variance of $N$ are given, respectively, by
\begin{align*}
\expec_t(N) &= -h'(t), \\
\var_t(N) &= \phantom{-}h''(t).
\end{align*}
\end{lem}

\begin{proof}
Denote $g(t)=G(e^{-t})=\sum_{n=0}^\infty r(n) e^{-nt}$, so that
\begin{align*}
g'(t) = -\sum_{n=0}^\infty n r(n) e^{-nt}, \qquad
g''(t) = \sum_{n=0}^\infty n^2 r(n) e^{-nt}.
\end{align*}
Using \eqref{eq:prob-dim-n}, it follows that
$$
h'(t) = (\log g)'(t) = \frac{g'(t)}{g(t)} = -\sum_{n=0}^\infty n \prob_t(N=n) = -\expec_t(N),
$$
and similarly,
\begin{align*}
h''(t) &= \frac{g''(t)g(t)-g'(t)^2}{g(t)^2}=
\frac{g''(t)}{g(t)}-\left(\frac{g'(t)}{g(t)}\right)^2
\hspace{120pt}
\\ &= \sum_{n=0}^\infty n^2 \prob_t(N=n) - (-\expec_t(N))^2 = \expec_t(N^2)-(\expec_t(N))^2=\var_t(N). 
\qedhere
\end{align*}
\end{proof}

\section{Saddle point analysis}

\label{sec:saddle-point}

Our next step is to perform a saddle point analysis of the generating function $G(x)$. This can be described in two equivalent languages. One standard description of the method commonly found in analysis textbooks is in terms of the complex-analytic idea that the coefficients $r(n)$ may be extracted from $G(x)$ as contour integrals, namely as
\begin{equation}
\label{eq:saddlepoint-contourint}
r(n) = \frac{1}{2\pi i}\oint\limits_{|z|=r_n} \frac{G(z)}{z^{n+1}}\,dz.
\end{equation}
The idea is then to choose the radius $r_n$ of the contour of integration in such a way that the bulk of the contribution to the integral comes from the environment of the point $z=r_n$ on the positive $x$-axis, and the nature of this contribution can be understood from the local asymptotic behavior of $G(x)$ near $x=1$.

An equivalent way to describe the saddle point technique (as applied to the present context) is in probabilistic terms, based on the relation \eqref{eq:prob-dim-n}. The idea is that we are free to choose the value of the parameter $t$ in the random representation model, and we will do so in a way that causes the random variable $N$ to have its mean value at (or near) $n$. Probabilistic intuition then predicts that $N$, being the sum of independent random variables, will have an approximately Gaussian distribution, and from this intuition an asymptotic formula for the left-hand side of \eqref{eq:prob-dim-n} (hence for $r(n)$) can be easily guessed. This guess, known as a ``local central limit theorem,'' can then be proved using Fourier inversion, which is formally equivalent to the contour integral \eqref{eq:saddlepoint-contourint}. Despite the equivalence of the two approaches from the point of view of mathematical analysis, they are not \emph{psychologically} equivalent, and we prefer the probabilistic approach for putting the analysis on a more conceptual footing.

Say that a sequence of positive numbers $(t_n)_{n=1}^\infty$ is an (approximate) \demph{saddle point sequence} for the random representation model if the asymptotic relation
\begin{equation} \label{eq:saddle-point}
\expec_{t_n}(N) = -h'(t_n) = n + O(n^{7/10})
\end{equation}
holds as $n\to\infty$. (As often happens with saddle point analysis, it is not crucial to identify the saddle point precisely, and for reasons that will become apparent shortly, an approximate expression with $O(n^{7/10})$ error will be sufficient for our needs.) Our first goal is to find solutions to \eqref{eq:saddle-point}. 
Referring to the asymptotic expansion \eqref{eq:hoft-asym-deriv}, we see that in order for $-h'(t_n)$ to grow to infinity, as it must according to \eqref{eq:saddle-point}, the sequence $t_n$ will decrease towards $0$. Considering only the leading term in \eqref{eq:hoft-asym-deriv} temporarily, we get an approximate equation
$$
\tfrac23\mu_1 t_n^{-5/3} \approx n,
$$
which can be easily solved to show that $t_n \approx \left(3n/2\mu_1\right)^{-3/5}$. That is, $t_n$ should decay roughly as a constant times $n^{-3/5}$, where the constant is $(\tfrac23\mu_1)^{3/5} =2X^2$ (with $X$ as defined in \eqref{eq:defx}). As it turns out, this expression is not sufficiently precise for our needs because the contribution from the second-order term in \eqref{eq:hoft-asym-deriv} (proportional to $t_n^{-3/2} \asymp n^{9/10}$) is too large, so we look for a more precise solution to \eqref{eq:saddle-point} whose form is given by the ansatz
\begin{equation}
\label{eq:saddle-point-ansatz}
t_n = \tau_1 n^{-3/5} - \tau_2 n^{-7/10} - \tau_3 n^{-4/5}, 
\end{equation}
with $\tau_1 = 2X^2$ and $\tau_2, \tau_3$ being constant coefficients whose value needs to be determined. Solving for $\tau_2$ and $\tau_3$ is now an amusing calculus exercise. Introduce symbols $q,u,v$, defined by 
\begin{equation} 
q=n^{-1/10}, \qquad u=\frac{\tau_2}{\tau_1}, \qquad v=\frac{\tau_3}{\tau_1}.
\label{eq:tofn-qexpansion-params}
\end{equation}
Then $t_n$ can be written as
\begin{equation} \label{eq:tofn-qexpansion}
t_n = \tau_1 q^6 (1-uq-vq^2) = \tau_1 q^6 (1-q(u+vq)).
\end{equation}
Now consider each of the first two leading terms in \eqref{eq:hoft-asym-deriv} with $t=t_n$. Using the second-order Taylor expansion $(1-z)^\alpha=1-\alpha z + \tfrac12 \alpha(\alpha-1)z^2+O(z^3)$ and expanding in powers of $q$, we have
\begin{align}
\tfrac23 \mu_1 t_n^{-5/3} &= \tfrac23 \mu_1 \Big(\tau_1 q^6\Big)^{-5/3} \Big(1-q(u+vq)\Big)^{-5/3} 
\nonumber \\ &= q^{-10} \Big(1-\left(-\tfrac53\right)q(u+vq)+\tfrac12\left(-\tfrac53\right)\left(-\tfrac83\right)q^2(u+vq)^2+O(q^3)\Big)
\nonumber \\ &=
q^{-10} + \left(\tfrac53 u\right)q^{-9} + \left(\tfrac53 v + \tfrac53\cdot\tfrac43 u^2\right)q^{-8} + O(q^{-7}),
\label{eq:saddle-point-term1}
\end{align}
and similarly,
\begin{align}
\tfrac12\mu_2t_n^{-3/2} &= \tfrac12\mu_2 \Big(\tau_1 q^6\Big)^{-3/2}
\Big(1-q(u+vq)\Big)^{-3/2}
\nonumber \\ &= \left( \tfrac12\mu_2 \tau_1^{-3/2}\right) q^{-9}\Big(
1 - \left(-\tfrac32\right)q(u+vq)+ \tfrac12\left(-\tfrac32\right)\left(-\tfrac52\right)q^2(u+vq)^2
\Big)
\nonumber \\ &= 
\left( \tfrac12\mu_2 \tau_1^{-3/2}\right) q^{-9}
+ \left(\tfrac32\cdot \tfrac12\mu_2 \tau_1^{-3/2} u \right)q^{-8} + O(q^{-7}).
\label{eq:saddle-point-term2}
\end{align}
The remaining terms are $O(t_n^{-1})=O(q^{-6})=O(n^{6/10})$, which is within our tolerance range for solving \eqref{eq:saddle-point}. Thus, adding up \eqref{eq:saddle-point-term1} and \eqref{eq:saddle-point-term2}, we get that
\begin{align*}
\expec_{t_n}(N) &= n + 
\left(\tfrac53 u+\tfrac12\mu_2 \tau_1^{-3/2}
\right)n^{9/10} \\ & \qquad + \left(
\tfrac53v+\tfrac53\cdot\tfrac43 u^2+\tfrac32\cdot\tfrac12\mu_2\tau_1^{-3/2} u
\right)n^{8/10} + O(n^{7/10}).
\end{align*}
Comparing this expansion to \eqref{eq:saddle-point} shows that the sequence $(t_n)_{n=1}^\infty$ will be a saddle point sequence provided that the equations
\begin{align*}
\tfrac53 u &= -\tfrac12\mu_2 \tau_1^{-3/2}, \\
\tfrac53v &= -\tfrac53\cdot\tfrac43 u^2-\tfrac32\cdot\tfrac12\mu_2\tau_1^{-3/2} u,
\end{align*}
hold, or equivalently, provided that $\tau_2, \tau_3$ take the values
\begin{align*}
\tau_2 &= -\tfrac{3}{10}\mu_2 \tau_1^{-1/2}, \\
\tau_3 &= -\tfrac43\tau_2^2\tau_1^{-1} - \tfrac{9}{20} \mu_2 \tau_1^{-3/2} \tau_2.
\end{align*}

After a small amount of further algebraic simplification, which we omit, we arrive at the following result.
 
\begin{lem}
Let the numbers $X,Y$ be defined by \eqref{eq:defx}--\eqref{eq:defy}. Define a sequence of numbers $(t_n)_{n=1}^\infty$ by
\begin{equation}
\label{eq:saddle-point-solution}
t_n = \tau_1 n^{-3/5} - \tau_2 n^{-7/10} - \tau_3 n^{-4/5} \qquad (n\ge1),
\end{equation}
where $\tau_1,\tau_2,\tau_3$ are the constants
\begin{align*}
\tau_1 &= 2X^2, \\[2pt]
\tau_2 &= \frac{3}{10}X^{-1}Y, \\[2pt]
\tau_3 &= \frac{3}{400}X^{-4}Y^2.
\end{align*}
Then $(t_n)_{n=1}^\infty$ is a saddle point sequence, that is, it satisfies the equation \eqref{eq:saddle-point}.
\end{lem}

It is worth noting that the expression \eqref{eq:saddle-point-solution} for the location of the saddle point requires considerably more precision than the case of Meinardus's theorem on partition asymptotics (discussed in Section~\ref{sec:asym-genfuns}). In Meinardus's analysis the saddle point could be defined using a single term, given as a constant times a power of $n$, analogous to the leading term $\tau_1 n^{-3/5}$ in \eqref{eq:saddle-point-solution}; see equation (6.2.15) on page~93 of \cite{andrews}. The reason is that Meinardus's theorem makes the relatively restrictive assumption that the Dirichlet series $\sum_m a_m m^{-s}$ (in the notation of \eqref{eq:meinardus-genfun}) can be analytically continued to the strip $\re(s)>\sigma_0$ for some $\sigma_0<0$, with only a single pole of order $1$ at some point $s=\alpha>0$. In our situation the relevant Dirichlet series is $2^s \wzeta(s)$, which has not one but \emph{two} separate poles with positive real parts, namely at $s=2/3$ and $s=1/2$. It is precisely this somewhat novel singularity structure that leads to the more involved saddle point analysis leading up to \eqref{eq:saddle-point-solution}, and which (as we shall see in the next few sections) ultimately causes the asymptotic formula \eqref{eq:su3-asym} for $r(n)$ to have the interesting structure it does.

\section{Saddle point analysis II: the local central limit theorem}

\begin{lem}
\label{lem:variance}
As $n\to\infty$, the variance of $N$ at the parameter value $t=t_n$ is given asymptotically by
\begin{equation} \label{eq:variance-asym}
\var_{t_n}(N) = \left(\tfrac56 X^{-2}\right)n^{8/5} + O(n^{3/2}).
\end{equation}
\end{lem}

\begin{proof}
Since, by Lemma~\ref{lem:formulas-expecvar}, the variance is given by $\var_{t_n}(N)=h''(t_n)$, substituting the values \eqref{eq:saddle-point-solution} into the asymptotic expansion \eqref{eq:hoft-asym-deriv2} gives the result after a quick calculation, which we omit.
\end{proof}

We are now ready to formulate our local central limit theorem for the random variable $N$, whose proof will be the heart of the saddle point analysis. Note that the saddle point equation \eqref{eq:saddle-point} identifies the expected value $\expec_{t_n}(N)$ to within an error of $O(n^{7/10})$. On the other hand, Lemma~\ref{lem:variance} shows that $N$ has standard deviation (relative to the parameter value $t=t_n$) of order $n^{4/5}=n^{8/10}$. The local central limit theorem is the statement that the distribution of $N$ is asymptotically Gaussian near its mean value---or equivalently, near the value $n$, since the two values are $O(n^{-1/10})$ standard deviations apart (this was the reason we required an error no greater than $O(n^{7/10})$ in \eqref{eq:saddle-point}). 
The precise statement is as follows.

\begin{thm}[Local central limit theorem for $N$]
\label{thm:local-clt}
The variable $N$ ``satisfies a local central limit theorem near its mean value.'' More precisely, as $n\to\infty$ we have the asymptotic relation
\begin{equation} \label{eq:local-clt}
\prob_{t_n}(N=n) = \frac{1+o(1)}{\sqrt{2\pi \var_{t_n}(N)}} = (1+o(1))\frac{\sqrt{3}X}{\sqrt{5\pi}}\cdot n^{-4/5}.
\end{equation}
\end{thm}

It is well-known that local central limit theorems for integer-valued random variables are inherently a more delicate and less robust phenomenon than the usual, non-local sort, and this case is no exception. The point is that the definition of $N$ involves a sum of independent, but not identically distributed, components $\tfrac12jk(j+k) X_{j,k}$, each of which takes its values in the sub-lattice $\big(\tfrac12jk(j+k)\big)\Z$ of $\Z$.
The proof of local CLTs of this type therefore invariably requires ruling out the potentially harmful influence of periodicities. (For example, in the most trivial example illustrating this principle, if the random variable we were looking at were a sum of random variables taking even values, then there would be zero probability for it to take on an odd value; however, even the absence of rigid or ``deterministic'' periodicities of this sort leaves room for more fuzzy periodicities of a probabilistic kind.) In our setting, this is at heart a number-theoretic claim, and the techniques of analytic number theory are the most appropriate to use to attack it. Specifically, it turns out that the modular transformation properties of the Jacobi theta functions are sufficient to prove the bounds we will need.

Before proceeding to read the proof of Theorem~\ref{thm:local-clt}, which we present in this section and the next one, the reader may wish to skip directly to Section~\ref{sec:proof-su3-asym}, where we show how Theorem~\ref{thm:local-clt} and the other results proved so far imply our theorem (Theorem~\ref{thm:su3-asym}) on the asymptotic enumeration of representations of $SU(3)$.

The proof of Theorem~\ref{thm:local-clt} begins with the characteristic function (a.k.a.\ Fourier-Stieltjes transform) of $N$, which we denote $\phi_t(u)$ and is given for $u\in \R$ by
\begin{align}
\phi_t(u) &= \expec_t(e^{i u N}) = \sum_{n=0}^\infty \prob_t(N=n)e^{i n u}
=\sum_{n=0}^\infty \frac{r(n)e^{-nt}}{G(e^{-t})} e^{inu}
\nonumber \\ &=\frac{1}{G(e^{-t})} \sum_{n=0}^\infty r(n)e^{-nt+inu}
= \frac{G(e^{-t+iu})}{G(e^{-t})}.
\label{eq:N-charfun}
\end{align}
(As before, the subscript $t$ emphasizes the dependence on the parameter $t$, which soon we will specialize to the saddle point values $t=t_n$.) Since $N$ is an integer-valued random variable, this is simply a Fourier series. Fourier inversion therefore gives that
\begin{equation} \label{eq:fourier-inversion}
\prob_t(N=n) = \frac{1}{2\pi} \int_{-\pi}^\pi \phi_t(u) e^{-inu}\,du. 
\end{equation}
Denote $\sigma_n = \left(\var_{t_n}(N)\right)^{1/2}$, the standard deviation of $N$. Setting $t=t_n$
and scaling the integration variable $u$ by $1/\sigma_n$ in \eqref{eq:fourier-inversion}, we get
$$
\prob_{t_n}(N=n)=\frac{1}{2\pi \sigma_n}\int_{-\pi \sigma_n}^{\pi \sigma_n} \phi_{t_n}(u/\sigma_n)e^{-inu/\sigma_n}\,du.
$$
Comparing this to the claim \eqref{eq:local-clt}, we see that it will be enough to prove that
\begin{equation} \label{eq:localclt-sufficient}
\int_{-\pi\sigma_n}^{\pi \sigma_n} 
\phi_{t_n}(u/\sigma_n)e^{-inu/\sigma_n}\,du \xrightarrow[n\to\infty]{} \sqrt{2\pi}.
\end{equation}

A key first step will be to prove pointwise convergence of the integrand to the characteristic function of the standard normal distribution.

\begin{thm}
\label{thm:nonlocal-clt}
For each $u\in\R$, we have
\begin{equation}
\label{eq:nonlocal-clt}
\phi_{t_n}(u/\sigma_n)e^{-inu/\sigma_n} \xrightarrow[n\to\infty]{} e^{-u^2/2}.
\end{equation}
\end{thm}

Note that Theorem~\ref{thm:nonlocal-clt} is equivalent to the claim that the normalized random variable $\hat{N}=\sigma_n^{-1}(N-\expec_{t_n}(N))$ converges in distribution to a standard normal random variable, that is, to a \emph{non-local} CLT. 

The proof of Theorem~\ref{thm:nonlocal-clt} involves not much more than a Taylor expansion, but one needs to make sure that the dependence on both the parameters $t$ and $u$ is carefully taken into account. We precede the main part of the proof with two technical lemmas. To motivate the statement of the first lemma, recall that if $X$ is a random variable with the geometric distribution $\operatorname{Geom}_0(p)$, where the parameter $0<p<1$ is expressed as $p=1-e^{-t}$, then the mean and variance of $X$ are given by $$ \expec(X) = \frac{1-p}{p}=\frac{e^{-t}}{1-e^{-t}}, \qquad \var(X) = \frac{1-p}{p^2}=\frac{e^{-t}}{(1-e^{-t})^2},
$$
and the 
characteristic function $\phi_X(u)=\expec(e^{iuX})$ of $X$ can be easily computed to be
$$ \phi_X(u) =\sum_{n=0}^\infty (1-e^{-t})e^{-nt} e^{inu} = \frac{1-e^{-t}}{1-e^{-t+iu}}. $$
In particular, if $\log z$ denotes the principal branch of the logarithm function, then
$\log \phi_X(u)$ has the Taylor expansion
$$ \log \phi_X(u) = i\, \expec(X) u -\tfrac12 \var(X)u^2 + O(u^3) \ \ \ \textrm{ as }u\to0. $$

\begin{lem}
\label{lem:romik2005}
Define a function 
$$
\eta(t,u) = \log\frac{1-e^{-t}}{1-e^{-t+iu}} - i \frac{e^{-t}}{1-e^{-t}}u + \tfrac12 \frac{e^{-t}}{(1-e^{-t})^2} u^2 \qquad (t>0, \ u\in\R).
$$
For some constant $C>0$, the bound
$$ |\eta(t,u)| \le C \frac{e^{-t}|u|^3}{(1-e^{-t})^3} $$
holds for all $u\in\R$ and $t>0$.
\end{lem}

Lemma~\ref{lem:romik2005} has been previously used in the study of integer partition asymptotics. See pp.~10--11 of \cite{romik} for the proof.

\begin{lem} 
\label{lem:double-integral}
The improper two-dimensional integral
\begin{equation} \label{eq:double-integral}
I=\int_0^\infty \int_0^\infty \left(\frac{xy(x+y)}{2}\right)^3\frac{\exp\left(-\tfrac12xy(x+y)\right)}{\left(1-\exp\left(-\tfrac12xy(x+y)\right)\right)^3}\,dx\,dy
\end{equation}
converges.
\end{lem}

\begin{proof}
Using the expansion $\frac{x}{(1-x)^3}=\sum_{n=1}^\infty \binom{n+1}{2}x^n$ to expand the integrand, we get that
\begin{align*}
I&=\sum_{n=1}^\infty \binom{n+1}{2} \int_0^\infty \int_0^\infty 
\left(\frac{xy(x+y)}{2}\right)^3 \exp\left(-\tfrac12nxy(x+y)\right)\,dx\,dy
\\ &\le \sum_{n=1}^\infty \frac{n^2}{n^{11/3}}
\int_0^\infty \int_0^\infty 
\left(\frac{ab(a+b)}{2}\right)^3 \exp\left(-\tfrac12ab(a+b)\right)\,da\,db
\\ &= \zeta\left(\tfrac53\right)
\int_0^\infty \int_0^\infty 
\left(\frac{ab(a+b)}{2}\right)^3 \exp\left(-\tfrac12ab(a+b)\right)\,da\,db.
\end{align*}
The fact that this last integral converges can be shown easily using the same change of variables used in the proof of Lemma~\ref{lem:hoft-asymbounds}.
\end{proof}

\begin{proof}[Proof of Theorem~\ref{thm:nonlocal-clt}]
Using \eqref{eq:N-charfun} and \eqref{eq:Gofx-def}, 
we represent $\phi_{t_n}(u/\sigma_n)e^{-inu/\sigma_n}$ as an infinite product. Taking the logarithm, we have that
\begin{align}
\log&\left[
\phi_{t_n}(u/\sigma_n)e^{-inu/\sigma_n}\right] = \log \frac{G(e^{-t_n+iu/\sigma_n})}{G(e^{-t_n})}-\frac{inu}{\sigma_n}
\nonumber \\ &= \sum_{j,k=1}^\infty \log \left(\frac{1-\exp\left(-\tfrac12 jk(j+k)t_n\right)}{1-\exp\left(\tfrac12 jk(j+k)(-t_n+iu/\sigma_n)\right)}\right)-\frac{inu}{\sigma_n}
\nonumber \\ &=
\sum_{j,k=1}^\infty 
\Bigg[
 \eta\left(\frac{jk(j+k)t_n}{2},\frac{jk(j+k)u}{2\sigma_n}\right)
\nonumber \\ & \qquad \qquad +i \frac{\tfrac12jk(j+k)\exp\left(-\tfrac12 jk(j+k)t_n\right)}{1-\exp\left(-\tfrac12 jk(j+k)t_n\right)}\cdot \frac{u}{\sigma_n} 
\nonumber \\ & 
\qquad \qquad
- \tfrac12
\frac{\left(\tfrac12jk(j+k)\right)^2\exp\left(-\tfrac12 jk(j+k)t_n\right)}{\left(1-\exp\left(-\tfrac12 jk(j+k)t_n\right)\right)^2}\cdot \frac{u^2}{\sigma_n^2}
\Bigg]-\frac{inu}{\sigma_n}
\nonumber \\ &=
\sum_{j,k=1}^\infty \eta\left(\frac{jk(j+k)t_n}{2},\frac{jk(j+k)u}{2\sigma_n}\right)
+i\left( \frac{\expec_{t_n}(N) - n}{\sigma_n}\right)u
\nonumber \\ & \qquad \qquad \quad
-\tfrac12 \var_{t_n}(N) \frac{u^2}{\sigma_n^2}
= -\tfrac12 u^2 + O(n^{-1/10}u) + R_n(u),
\label{eq:log-charfun}
\end{align}
where we denote
$ R_n(u)=\sum_{j,k=1}^\infty \eta\left(\frac{jk(j+k)t_n}{2},\frac{jk(j+k)u}{2\sigma_n}\right).$
Invoking Lemma~\ref{lem:romik2005}, this quantity can be bounded as
\begin{align*}
|R_n(u)| &\le \sum_{j,k=1}^\infty \left|\eta\left(\frac{jk(j+k)t_n}{2},\frac{jk(j+k)u}{2\sigma_n}\right)\right| 
\\ &\le C \frac{|u|^3}{\sigma_n^3} \sum_{j,k=1}^\infty
\left(\frac{jk(j+k)}{2}\right)^3 
\frac{\exp\left(-\tfrac12 jk(j+k)t_n\right)}{\left(1-\exp\left(-\tfrac12 jk(j+k)t_n\right)\right)^3}
\\ &= C \frac{|u|^3}{\sigma_n^3 t_n^{11/3}} \sum_{j,k=1}^\infty
\left(\frac{t_n jk(j+k)}{2}\right)^3 
\frac{\exp\left(-\tfrac12 jk(j+k)t_n\right)}{\left(1-\exp\left(-\tfrac12 jk(j+k)t_n\right)\right)^3}
\cdot t_n^{2/3}.
\end{align*}
In this expression, the prefactor 
$C |u|^3\sigma_n^{-3} t_n^{-11/3}$
is of order $O(n^{-1/5} |u|^3)$, and the sum is a Riemann sum (with $\Delta x=\Delta y = t_n^{1/3}$) for the double integral \eqref{eq:double-integral}. Since that integral is finite by Lemma~\ref{lem:double-integral}, the Riemann sum converges to it and in particular is bounded. This together with the above computations establishes that
$$
\log\left[
\phi_{t_n}(u/\sigma_n)e^{-inu/\sigma_n}\right] \xrightarrow[n\to\infty]{} -\tfrac12 u^2, 
$$
which implies \eqref{eq:nonlocal-clt} and therefore finishes the proof.
\end{proof}

The computations in the proof above also imply a useful bound on $|\phi_{t_n}(u/\sigma_n)|$, which we summarize in the following lemma.

\begin{lem}
\label{lem:easy-tailbound}
There exists a constant $D>0$ such that the bound
\begin{equation} \label{eq:easy-tailbound} 
|\phi_{t_n}(u/\sigma_n)| \le e^{-u^2/4}
\end{equation}
holds for $u$ satisfying $|u|\le D n^{1/5}$.
\end{lem}

\begin{proof}
By our earlier computations, we have that
$$ |\phi_{t_n}(u/\sigma_n)|=\exp\left[ \re \Big(\log \phi_{t_n}(u/\sigma_n) \Big)\right]
\le \exp\left( -\tfrac12 u^2 + C_1 n^{-1/5}|u|^3 \right)
$$
for some constant $C_1>0$. 
(Note that the term $i(\expec_{t_n}(N)-n)u/\sigma_n$ in \eqref{eq:log-charfun} is a pure imaginary number and therefore only contributes a phase factor to $\phi_{t_n}(u/\sigma_n)$.)
Restricting $u$ to satisfy $|u|\le D n^{1/5}$ for some sufficiently small $D>0$ guarantees that $-\tfrac12u^2 + C_1n^{-1/5}|u|^3 \le -\tfrac14 u^2$, so in that range we get the bound \eqref{eq:easy-tailbound}.
\end{proof}

\section{Saddle point analysis III: the Jacobi theta functions}

In order to deduce \eqref{eq:localclt-sufficient} (and hence Theorem~\ref{thm:local-clt}) from Theorem~\ref{thm:nonlocal-clt}, we will need the ``easy'' tail bound \eqref{eq:easy-tailbound} together with a much more delicate bound on $|\phi_{t_n}(u/\sigma_n)|$ that can be shown to hold outside the range $|u|\le Dn^{1/5}$. The goal of this section is to prove this more difficult bound, given in the following theorem.

\begin{thm}
\label{thm:hard-tailbound}
There exists a constant $E>0$ such that the bound
\begin{equation} \label{eq:hard-tailbound} 
|\phi_{t_n}(u/\sigma_n)| \le \exp\left(-E n^{3/10}\right)
\end{equation}
holds for $u$ satisfying $D n^{1/5} \le |u|\le \pi\sigma_n$, where $D$ is the constant from Lemma~\ref{lem:easy-tailbound}.
\end{thm}

As preparation for the proof, replace $u/\sigma_n$ by $u$, so in view of the asymptotics \eqref{eq:variance-asym} of the variance, the statement we need to prove becomes that 
\begin{equation}\label{eq:phitn-bound}
|\phi_{t_n}(u)|\le \exp\left(-En^{3/10}\right)
\end{equation}
for some constant $E>0$, all sufficiently large $n$, and all $u$ such that
$D'n^{-3/5}\le |u|\le \pi$, where $D'=2D/(\sqrt{5/6} X^{-1})$. More generally, we will prove that for any $\epsilon>0$, there exist numbers $\beta,\delta>0$ 
such that if $0<t<\beta$ and $\epsilon t \le |u| \le \pi$ then
\begin{equation}
\label{eq:hard-tailbound-reduction}
|\phi_t(u)| \le \exp\left(-\delta t^{-1/2}\right).
\end{equation}
The bound \eqref{eq:phitn-bound} then follows by substituting $t=t_n$ (refer to \eqref{eq:saddle-point-solution}).

Now observe that
\begin{align*}
\log \phi_t(u) &= \log G(e^{-t+iu})-\log G(e^{-t}) = h(t-iu)-h(t)
\\ &=
\sum_{m=1}^\infty \frac{1}{m}(f(m(t-iu)) - f(mt)).
\end{align*}
By the definition of $f(t)$ in \eqref{eq:foft-def}, clearly $\re(f(t-iu))\le \re(f(t))$ for any $u\in\R, t>0$, so
\begin{align*}
\re\big(\log &\phi_t(u)\big) \le 
\sum_{m=1}^\infty \frac{1}{m}
\re(f(m(t-iu))-f(mt)) \le  \re(f(t-iu)-f(t)) 
\\ &\le
\re \left[ \sum_{j,k=1}^\infty \Big(\exp\left(-\tfrac12 jk(j+k)(t-iu)\right)
-\exp\left(-\tfrac12 jk(j+k)t\right) \Big)\right].
\end{align*}
In this last expression, the real part of each of the summands is nonpositive, so we can omit as many of the terms as we like and still get an upper bound. We focus on the subset $j=1$ of the summation range, leading to the bound
\begin{align}
\re\left(\log \phi_t(u)\right) & \le \re\left[ \sum_{k=1}^\infty 
\exp\left(-\tfrac12 k(k+1)(t-iu)\right) 
-
\sum_{k=1}^\infty 
\exp\left(-\tfrac12 k(k+1)t\right) \right]
\nonumber \\ &
= \re(\lambda(t-iu))-\lambda(t),
\label{eq:log-charfun-rebound}
\end{align}
where we denote
$$ \lambda(z) = \sum_{k=0}^\infty \exp\left(-\tfrac12 k(k+1)z \right). $$

\begin{prop}
\label{prop:lambda-bounds}
For any $\epsilon>0$, there exist numbers $\beta, \gamma>0$ and \hbox{$0<\rho<1$} such that
for all $t,u$ satisfying $0<t<\beta$ and $\epsilon t \le |u| \le \pi$, the bounds
\begin{equation}
\label{eq:lambda-bounds}
\lambda(t) \ge \gamma t^{-1/2}, \qquad \frac{|\lambda(t-iu)|}{\lambda(t)} \le 1-\rho, \end{equation}
hold.
\end{prop}

Proposition~\ref{prop:lambda-bounds} is enough to prove \eqref{eq:hard-tailbound-reduction}, and hence Theorem~\ref{thm:hard-tailbound}, since \eqref{eq:log-charfun-rebound} and \eqref{eq:lambda-bounds} imply (under the assumed conditions on $t,u$) that
\begin{align*}
|\phi_t(u)| &= \exp\left(\re(\log \phi_t(u))\right) \le \exp\big(-\lambda(t)+(1-\rho)\lambda(t)\big)
\\ &= \exp(-\rho \lambda(t)) \le 
\exp\left(-\rho\gamma \, t^{-1/2}\right).
\end{align*}

To prove the proposition, we will use the fact that $\lambda(z)$ can be related to the Jacobi theta functions. Define
\begin{align*}
\theta_2(q) &= \sum_{k=-\infty}^\infty q^{(k+1/2)^2} = 2q^{1/4} \sum_{k=0}^\infty q^{k(k+1)}, 
& J_2(z) = \theta_2(e^{i\pi z}), \\
\theta_3(q) &= \sum_{k=-\infty}^\infty q^{k^2},  & J_3(z) = \theta_3(e^{i\pi z}), \\
\theta_4(q) &= \sum_{k=-\infty}^\infty (-1)^k q^{k^2},
&
J_4(z)=  \theta_4(e^{i\pi z}).
\end{align*}
The functions $J_2(z), J_3(z), J_4(z)$ are holomorphic functions in the upper half plane $\mathbb{H}=\{x+iy\,:\,y>0\}$. They are known to satisfy the modular transformation relations
\begin{align}
\label{eq:modularity1} J_2(z) &= (-iz)^{-1/2} J_4(-1/z), \\
\label{eq:modularity2} J_3(z) &= (-iz)^{-1/2} J_3(-1/z),
\end{align}
where $w\mapsto w^{1/2}$ is the principal branch of the square root function;
see \cite[p.~181]{rauch-lebowitz}.
In particular, note that for real $y>0$, $J_3(i y)=1+O(e^{-\pi y})$ and $J_4(iy)=1+O(e^{-\pi y})$ as $y\to\infty$. Therefore \eqref{eq:modularity1} and \eqref{eq:modularity2} imply that as $y\to 0$ we have
\begin{align}
\label{eq:jacobi1-asym}
J_2(iy) &= \frac{1}{\sqrt{y}} J_4(i/y) = \frac{1}{\sqrt{y}}\left(1+O\left(e^{-\pi/y}\right)\right), \\
\label{eq:jacobi2-asym}
J_3(iy) &= \frac{1}{\sqrt{y}} J_3(i/y) = \frac{1}{\sqrt{y}}\left(1+O\left(e^{-\pi/y}\right)\right).
\end{align}
Furthermore, we also have trivially that $|J_4(x+iy)|\le J_3(iy)$.

Now we can relate the previous discussion regarding $\lambda(t-iu)$ to the functions $J_2(z), J_3(z), J_4(z)$, by noting that
\begin{align*}
\lambda(t-iu) &= \sum_{k=0}^\infty\exp\left(-\tfrac12k(k+1)(t-iu)\right) 
= \tfrac12 e^{t/8} e^{-iu/8} J_2\left(\frac{u+it}{2\pi}\right).
\end{align*}
In particular, when $t\searrow 0$ we have the bound $\lambda(t)=\tfrac12 e^{t/8}J_2(it/2\pi)=\tfrac12\left(1+O\left(e^{-2\pi^2/t}\right)\right)\sqrt{2\pi/t} \ge \gamma t^{-1/2}$, which was the first claim of Proposition~\ref{prop:lambda-bounds}. For the second claim, by modifying the constants we can forget about the $1/\pi$ factor, and prove instead that 
$$ \frac{|\lambda(\pi(t-iu))|}{\lambda(\pi t)} = \frac{\left|J_2\left(\frac{u+it}{2}\right)\right|}{J_2(it/2)}
$$
is bounded away from $1$ (more precisely, is $<\rho$ for some number $\rho<1$), under the assumption that: 
\begin{enumerate}
\item $\epsilon t\le |u|\le 1$ where $\epsilon>0$ is given; 
\item $0<t<\beta$ where we are free to fix $\beta$ as small as we please.
\end{enumerate}
(Naturally, $\beta$ and $\rho$ will depend on $\epsilon$.) Observe therefore that, thanks to \eqref{eq:modularity1}, we have
\begin{align}
\nonumber 
\frac{\left|J_2\left(\frac{u+it}{2}\right)\right|}{J_2(it/2)} &= \frac{(t/2)^{1/2}}{\left|(u+it)/2\right|^{1/2}}
\left(1+O\left(e^{-2\pi/t}\right)\right) 
\left| J_4\left(-\frac{2u}{u^2+t^2}+i\frac{2t}{u^2+t^2}\right)\right|
\\ &\le
\left(1+O\left(e^{-2\pi/t}\right)\right)  \left( \frac{t^2}{u^2+t^2}
\right)^{1/4} J_3\left(i\frac{2t}{u^2+t^2}\right).
\label{eq:bound-from1}
\end{align}
To bound this expression away from $1$, we divide into three cases:

\paragraph{Case 1.} Assume that $\epsilon t\le |u| \le t^{3/4}$. In this case, 
$$ \left(\frac{t^2}{t^2+u^2}\right)^{1/4} \le \left(\frac{t^2}{t^2+\epsilon^2 t^2}\right)^{1/4} = \left(\frac{1}{1+\epsilon^2}\right)^{1/4}. $$
The quantity on the right-hand side is smaller than $1$, and we denote it $1-\rho_1$. Second, making the further assumption that $t<1$ (which simply puts the constraint $\beta<1$ on our ultimate choice of $\beta$), we have
$$ \frac{2t}{u^2+t^2} \ge \frac{2t}{t^{3/2}+t^2} \ge \frac{2t}{2t^{3/2}} = \frac{1}{\sqrt{t}}, $$
so that (by the bound for $J_3(iy)$ as $y\to\infty$) $J_3\left(i\frac{2t}{u^2+t^2}\right) = 1+O(e^{-\pi t^{-1/2}})$ as $t\to 0$. Combining the last two estimates, we get that
the right-hand side of \eqref{eq:bound-from1} is bounded by
$$\left(1+O\left(e^{-2\pi/t}\right)\right)\left(1+O\left(e^{-\pi/\sqrt{t}}\right)\right)\times(1-\rho_1), $$
which is bounded by $1-\rho_1/2$ if $\beta$ is assumed to be small enough.

\paragraph{Case 2.} Assume that $t^{3/4} \le |u| \le M t^{1/2}$, where $M$ is some large enough constant whose value will be specified shortly. In this case, note that, again using the assumption that $t<1$, we have 
$$
\frac{2t}{u^2+t^2} \ge \frac{2t}{M^2 t+t^2} \ge \frac{2t}{(M^2+1) t}=\frac{2}{M^2+1}. $$
The function $s \mapsto J_3(is)$ of a positive real variable $s$ is monotone decreasing. So, after fixing $M$, we get that $J_3\left(i\frac{2t}{u^2+t^2}\right)$ is bounded by $L:=J_3\left(\frac{2i}{M^2+1}\right)$ in the range of values under discussion.
Furthermore,
we have
$$ \left(\frac{t^2}{t^2+u^2}\right)^{1/4} \le \left(\frac{t^2}{t^2+t^{3/2}}\right)^{1/4} \le \left(\frac{t^2}{t^{3/2}}\right)^{1/4}=t^{1/8}. $$
By restricting $\beta$ (and therefore $t$) to be less than $\left(\tfrac12 L^{-1}\right)^8$, we ensure that 
$$
\left(\frac{t^2}{t^2+u^2}\right)^{1/4} J_3\left(i\frac{2t}{t^2+u^2}\right) \le t^{1/8} L \le 1/2,
$$
which, when combined with \eqref{eq:bound-from1}, proves the claim in this case.

\paragraph{Case 3.} Finally, assume that $M t^{1/2} \le |u| \le 1$ (and recall that we are still free to specify the value of $M$). In this case, we have
$$
\frac{2t}{t^2+u^2} \le \frac{2t}{t^2+M^2 t} \le \frac{2t}{M^2 t} = \frac{2}{M^2}.$$ 
Using \eqref{eq:jacobi2-asym}, we see that if we assume $M$ to be large enough, this forces $y=2t/(t^2+u^2)$ to be such that $J_3(iy)\le 1.01/\sqrt{y}$. Selecting such a value of $M$, 
and restricting $\beta$ to be small enough to ensure that the $\left(1+O\left(e^{-2\pi/t}\right)\right)$ term in \eqref{eq:bound-from1} is also $<1.01$,
we get therefore that
the right-hand side of \eqref{eq:bound-from1} is bounded by
\begin{align*}
1.03 \left(\frac{t^2}{t^2+u^2}\right)^{1/4} \frac{(t^2+u^2)^{1/2}}{(2t)^{1/2}} &= \frac{1.03}{\sqrt{2}}(t^2+u^2)^{1/4} 
\le 0.8 (1+t^2)^{1/4}.
\end{align*}
If we assume further that $\beta<1/2$, this is bounded by $0.9$, which proves the claim in this case.

\bigskip
Cases 1, 2 and 3 above cover all the possibilities. We have therefore proved Proposition~\ref{prop:lambda-bounds}, which as we have seen implies Theorem~\ref{thm:hard-tailbound}.
\qed

\begin{proof}[Proof of Theorem~\ref{thm:local-clt}]
Decompose the integral in \eqref{eq:localclt-sufficient} as
\begin{align}
\int_{-\pi\sigma_n}^{\pi \sigma_n} 
\phi_{t_n}(u/\sigma_n) & e^{-inu/\sigma_n}\,du 
=
\int_{-D n^{1/5}}^{D n^{1/5}} 
\phi_{t_n}(u/\sigma_n)e^{-inu/\sigma_n}\,du 
\nonumber \\ &+
\int_{[-\pi\sigma_n,\pi \sigma_n]\setminus[-Dn^{1/5},Dn^{1/5}]}
\phi_{t_n}(u/\sigma_n)e^{-inu/\sigma_n}\,du.
\label{eq:localclt-sufficient-decomp}
\end{align}
Here, the first integral converges to $\int_{-\infty}^\infty e^{-u^2/2}\,du=\sqrt{2\pi}$ by 
Theorem~\ref{thm:nonlocal-clt}, Lemma~\ref{lem:easy-tailbound} and the dominated convergence theorem. The second integral can be seen using Theorem~\ref{thm:hard-tailbound}
to satisfy the bound
$$
\left|
\int_{[-\pi\sigma_n,\pi \sigma_n]\setminus[-Dn^{1/5},Dn^{1/5}]}
\phi_{t_n}(u/\sigma_n)e^{-inu/\sigma_n}\,du 
\right|
\le 2\pi \sigma_n \exp\left(-E n^{3/10}\right),
$$
and therefore converges to $0$. It follows that the left-hand side of \eqref{eq:localclt-sufficient-decomp} converges to $\sqrt{2\pi}$, which is precisely \eqref{eq:localclt-sufficient}. As we have seen, that result implies Theorem~\ref{thm:local-clt}.
\end{proof}

\section{Proof of Theorem~\ref{thm:su3-asym}}

\label{sec:proof-su3-asym}

Rewrite \eqref{eq:prob-dim-n} in the form
\begin{equation}
\label{eq:rofn-formula}
r(n) = G(e^{-t_n}) \exp(n t_n) \prob_{t_n}(N=n).
\end{equation}
The last factor on the right-hand side is given asymptotically by \eqref{eq:local-clt}, so the proof of Theorem~\ref{thm:su3-asym} reduces to writing down explicitly the asymptotics of the first two factors, using Theorem~\ref{thm:hoft-asym} and the definition of $t_n$ in \eqref{eq:saddle-point-solution}, and then combining the pieces. First, we have
that
\begin{equation} 
\label{eq:final1}
\exp(nt_n) = \exp\left( \tau_1 n^{2/5} - \tau_2 n^{3/10} - \tau_3 n^{1/5} \right).
\end{equation}
Second, to expand $G(e^{-t_n})$, recall that $t_n$ can be expressed in the form $t_n=\tau_1 q^6(1-q(u+vq))$ as in \eqref{eq:tofn-qexpansion}, in terms of the parameters $u=\tau_2/\tau_1$, $v=\tau_3/\tau_1$, $q=n^{-1/10}$ defined in \eqref{eq:tofn-qexpansion-params}. We then have that
\begin{align}
G(e^{-t_n}) &=
\exp\bigg(
\mu_1 t_n^{-2/3} + \mu_2 t_n^{-1/2} - \tfrac13 \log t_n 
\nonumber \\ & \qquad\qquad\qquad\qquad\qquad + \wzeta'(0) + \tfrac13\log 2 + O(n^{-3/10})
\bigg),
\label{eq:final2}
\end{align}
and each of the non-constant summands in the expansion can be evaluated asymptotically using a Taylor expansion, similar to the derivation in Section~\ref{sec:saddle-point}. 
The computation is a bit more tedious, since this time we need to consider the expansion up to the fourth order to attain the desired precision up to a $o(1)$ error. 

Start with the term proportional to $t_n^{-2/3}$. We have (compare with \eqref{eq:saddle-point-term1})
\begin{align*}
t_n^{-2/3} &= \left(\tau_1 q^6\right)^{-2/3} \left(1-q(u+vq)\right)^{-2/3}
\\ &
= \tau_1^{-2/3} q^{-4} \Bigg( 1+ \frac23 q(u+vq)
+\frac12\cdot\frac{2\cdot5}{3\cdot3}q^2(u+vq)^2 
\\ & \qquad 
+
\frac16\cdot \frac{2\cdot5\cdot8}{3\cdot3\cdot3}q^3(u+vq)^3
+\frac{1}{24}\cdot\frac{2\cdot5\cdot8\cdot11}{3\cdot3\cdot3\cdot3} q^4(u+vq)^4+O(q^5)
\Bigg)
\\ &=
\tau_1^{-2/3} \Bigg(
q^{-4} +
\left(\frac23 u \right) q^{-3} +
\left(\frac23v + \frac59 u^2 \right) q^{-2}
\\ & \qquad \qquad +
\left(\frac{10}{9}uv+\frac{2\cdot5\cdot8}{6\cdot3\cdot3\cdot3}u^3 \right) q^{-1} 
\\ & \qquad \qquad +
\left(\frac59v^2+\frac{2\cdot5\cdot8}{6\cdot3\cdot3\cdot3} 3u^2 v + \frac{2\cdot5\cdot8\cdot11}{24\cdot3\cdot3\cdot3\cdot3} u^4 \right) q^0 +
O(q)
\Bigg).
\end{align*}
Similarly, for the $t_n^{-1/2}$ term we have
\begin{align*}
t_n^{-1/2} &= \left(\tau_1 q^6\right)^{-1/2} \left(1-q(u+vq)\right)^{-1/2}
\\ &
= \tau_1^{-1/2} q^{-3} \Bigg( 1+ \frac12 q(u+vq)
+\frac12\cdot\frac{1\cdot3}{2\cdot2}q^2(u+vq)^2 
\\ & \qquad \qquad\qquad\qquad\qquad\qquad\ 
+
\frac16\cdot \frac{1\cdot3\cdot5}{2\cdot2\cdot2}q^3(u+vq)^3+O(q^4)
\Bigg)
\\ &=
\tau_1^{-1/2} \Bigg(
q^{-3} +
\left(-\frac12 u \right) q^{-2} +
\left(\frac12v + \frac{1\cdot3}{2\cdot 2\cdot 2} u^2 \right) q^{-1}
\\ & \qquad \qquad +
\left(\frac{1\cdot3}{2\cdot2\cdot2}2uv+\frac{1\cdot3\cdot5}{6\cdot2\cdot2\cdot2}u^3 \right) q^{0} +
O(q)
\Bigg).
\end{align*}
The contribution of the $\log t_n$ term is simpler to understand, since after exponentiating it becomes
\begin{align*}
\exp\left(-\tfrac13\log t_n\right) &= t_n^{-1/3} =
\left(\tau_1 q^6\right)^{-1/3} (1+q(u+vq))^{-1/3} \\ &= \tau_1^{-1/3} q^{-2} (1+ O(q)),
\end{align*}
and the error term $1+O(q)=1+O(n^{-1/10})$, being a multiplicative term outside the exponential sum, is sufficient for our purposes.

Substituting the expansions above into \eqref{eq:final2} and collecting terms in powers of $n^{1/10}$, we conclude that $G(e^{-t_n})$ has an asymptotic expansion of the form
\begin{align}
\nonumber
G(e^{-t_n}) &= \tau_1^{-1/3}n^{1/5} \exp\Big( B_1 n^{2/5} + B_2 n^{3/10} + B_3 n^{1/5} + B_4 n^{1/10}
\\ & \ \ \ \qquad \ \ \ \ \qquad\qquad\ \ \ \   + B_5 + \wzeta'(0)+\tfrac13\log 2 +O(n^{-1/10})
\Big),
\label{eq:final3}
\end{align}
where $B_1,B_2,B_3,B_4$ are constants, having the elaborate definitions
\begin{align*}
B_1 &= \mu_1 \tau_1^{-2/3}
,\\
B_2 &= \frac23 \mu_1 \tau_1^{-2/3} u + \mu_2 \tau_1^{-1/2}
,\\
B_3 &= \frac23 \mu_1 \tau_1^{-2/3} v + \frac59 \mu_1\tau_1^{-2/3} u^2
+\frac12 \mu_2 \tau_1^{-1/2} u
,\\
B_4 &= \frac{10}{9} \mu_1\tau_1^{-2/3} u v + \frac{40}{81} \mu_1\tau_1^{-2/3} u^3
+\frac12 \mu_2\tau_1^{-1/2} v + \frac38 \mu_2\tau_1^{-1/2} u^2,
\\
B_5 &=
\frac59 \mu_1\tau_1^{-2/3} v^2 + \frac{40}{27} \mu_1\tau_1^{-2/3}u^2 v + \frac{110}{243} \mu_1\tau_1^{-2/3} u^4 
\\ & \qquad \qquad\qquad\ \ \ \ 
+ \frac34 \mu_2\tau_1^{-1/2}uv +\frac{5}{16} \mu_2\tau_1^{-1/2}u^3.
\end{align*}
The definitions of $B_1, B_2, B_3, B_4, B_5$ can be further simplified. The details of this purely routine algebraic simplification can be found in the companion package \cite{romik-mathematica}. This results in the more pleasing expressions
\begin{align*}
B_1 &= 3X^2,
\\
B_2 &=-\frac{7}{10} X^{-1}Y,
\\
B_3 &=-\frac{3}{100} X^{-4}Y^2,
\\
B_4 &= -\frac{11}{3200} X^{-7}Y^3,
\\
B_5 &= -\frac{1}{2560} X^{-10}Y^4.
\end{align*}
We are finally ready to derive \eqref{eq:su3-asym}. Multiplying \eqref{eq:local-clt}, \eqref{eq:final1}, and \eqref{eq:final3}, we get using \eqref{eq:rofn-formula} that
\begin{align*}
r(n) &= (1+o(1))\left(
\tau_1^{-1/3} \frac{\sqrt{3}}{\sqrt{5\pi}}X \exp\left(B_5 +\wzeta'(0)+\tfrac13\log 2\right)
\right)
n^{-3/5} 
\\ & \quad \times \exp\left( (B_1+\tau_1)n^{2/5}+(B_2-\tau_2)n^{3/10}+(B_3-\tau_3)n^{1/5}+B_4 n^{1/10}
\right).
\end{align*}
This expansion is precisely of the form \eqref{eq:su3-asym}, with (one can easily check) our claimed values \eqref{eq:A1formula}--\eqref{eq:A4formula} for the constants $A_1,A_2,A_3,A_4$ and the value \eqref{eq:mult-const-altdef} for $K$, which in conjunction with Theorem~\ref{thm:borwein-dilcher} implies \eqref{eq:multiplicativeconst-def}. The proof of Theorem~\ref{thm:su3-asym} is complete.
\qed

\section{Final comments}

\label{sec:final-comments}

\subsection{Additional summation identities}

Our proof of Theorem~\ref{thm:eisenstein} relied on computer experimentation to discover the coefficients $\alpha_{n,j}$ in \eqref{eq:coeffs-guess1}. (Moreover, as we commented in Section~\ref{sec:bernoulli}, the coefficients $\beta_{n,k}$ could also be guessed in principle although we happened to know them in advance.) This suggests the possibility of discovering and proving additional identities, and perhaps even automating the process leading to such results. For example, by slightly generalizing the system of linear equations \eqref{eq:lineqs} to cover different mod $6$ congruence classes, we were able to discover empirically the following additional identities:
\begin{align}
\frac{B_{6n-2}}{6n-2} &=- \frac{1}{2(6n+1)}\binom{4n}{2n} \sum_{k=1}^n \binom{2n}{2k-1} P(n,k) \frac{B_{2n+2k-2}}{2n+2k-2} \cdot \frac{B_{4n-2k}}{4n-2k}, 
\label{eq:mod6-identity1}
\\
\frac{B_{6n}}{6n} &= -\frac{2}{3(6n+1)}\cdot \frac{(4n+1)!}{(2n)!^2} \sum_{k=0}^n \binom{2n}{2k} Q(n,k) \frac{B_{2n+2k}}{2n+2k} \cdot \frac{B_{4n-2k}}{4n-2k},
\label{eq:mod6-identity2}
\end{align}
where
\begin{align*}
P(n,k) &= (2n-1)^2-4(k-1)(n-k), \\[2pt]
Q(n,k) &= \frac{(n^2-kn+k^2)(4(n^2-kn+k^2)-1)-6kn(n-k)}{n(2k+1)(2n-2k+1)}.
\end{align*}
Since the discovery was based on the summation technique of Sections~\ref{sec:bernoulli} and \ref{sec:eisenstein}, these identities should arise as specializations of similar identities for the Eisenstein series. We did not work out the details of a proof, but our methods ought to apply.

\subsection{Open problems}

This work saw the fruitful meeting of ideas from complex analysis, number theory, asymptotic analysis, combinatorics, probability theory, and representation theory. Our results suggest several open problems that remain ripe for further investigation. Among them are the following:

\begin{enumerate}
\item Find additional lower order terms in the asymptotic expansion for $r(n)$.
\item Prove analogous versions of our results for other connected simple compact Lie groups and simple Lie algebras.
\item Find a direct proof of the result \eqref{eq:wzeta-trivialzeros} that $\wzeta(-n)=0$ for integer $n\ge1$ without relying on the Bernoulli number summation identity~\eqref{eq:bernoulli}, thereby deducing \eqref{eq:bernoulli} in a new way.
\item Use ideas from the theory of modular forms to reprove \eqref{eq:eisenstein}\footnote{Update added in revision: following our description of this problem in the original version of this paper, a proof of \eqref{eq:eisenstein} based on known results about modular forms was found recently by Mertens and Rolen \cite{mertens-rolen}.} and to discover and classify other identities with a similar flavor. Do modular forms give additional insight into the properties and significance of the $SU(3)$ zeta function $\wzeta(s)$?

\end{enumerate}

\end{document}